\documentclass[12pt]{article}
\usepackage{amsthm,amsmath}
\usepackage{amscd}
\usepackage{hyperref}
\usepackage[all]{xypic}
\usepackage{mathtools}
\usepackage{xcolor}
\usepackage[utf8]{inputenc}
\usepackage{amssymb,mathrsfs}
\usepackage{hyperref}
\usepackage{times}
\usepackage{inputenc}
\usepackage{lipsum}
\usepackage{titlesec}
\titlelabel{\thetitle. \quad}
\usepackage{cleveref}
\usepackage[toc]{appendix}
\usepackage{apptools}
\AtAppendix{\counterwithin{lemma}{section}}

\newcommand\blfootnote[1]{%
    \begingroup
    \renewcommand\thefootnote{}\footnote{#1}%
    \addtocounter{footnote}{-1}%
    \endgroup
}

\DeclareUnicodeCharacter{0308}{HERE!HERE!}

\newtheorem{lemma}{Lemma}[subsubsection]
\newtheorem{corollary}[lemma]{Corollary}

\newtheorem{theorem}[lemma]{Theorem}
\newtheorem{proposition}[lemma]{Proposition}
\newtheorem{definition}[lemma]{Definition}
\newtheorem{remark}[lemma]{Remark}
\newtheorem{example}[lemma]{Example}

\newcounter{lemmaaux}
\usepackage{chngcntr}
\def\includegraphics{}

\def\inf{\operatorname{inf}}
\newcommand{\vertiii}[1]{{\left\vert\kern-0.25ex\left\vert\kern-0.25ex\left\vert #1 
    \right\vert\kern-0.25ex\right\vert\kern-0.25ex\right\vert}}

\advance\headheight1.15pt
\begin{document}

\title{\fontsize{15}{0}\selectfont 
 K\"othe-Herz Spaces: The Amalgam-Type Spaces of Infinite Direct Sums   \\ 
}
\author{\fontsize{11}{0}\selectfont 
M. Ashraf Bhat$^1$, Pawe{\l}  Kolwicz$^2$ and G. Sankara Raju Kosuru$^*$ \\
\fontsize{10}{0}\textit{{$^{1,*}$-Department of Mathematics, Indian
Institute of Technology Ropar, Rupnagar-140001, Punjab, India.}}\\
\fontsize{10}{0}\textit{{$^{2}$-}Poznan University of Technology, Institute
of Mathematics, Piotrowo 3A, 60-965 Pozna\'{n}, Poland.}\\
}
\date{}
\maketitle

\thispagestyle{empty}

\begin{center}
\textbf{\underline{ABSTRACT}}
\end{center}
In this paper, we introduce a class of function spaces called K\"othe-Herz spaces $E(\mathcal{X})$. These spaces are similar to amalgam spaces and are characterized by a local component given by a countable family $\mathcal{X}=\left( X_{\alpha }\right) _{\alpha \in I}$ of quasi-normed function spaces, and a global component $E$, which is a quasi-normed sequence space. We investigate various geometric and topological properties inherited by $E(\mathcal{X})$ from its components, such as their completeness, duality, order continuity, ideal and Fatou properties, in an abstract setting. In addition, we provide a Banach function space characterization for $E(\mathcal{X})$, which allows us to understand its structure and behavior more deeply. Furthermore, by appropriate amalgamation of Lorentz spaces (Orlicz spaces) and Lebesgue sequence spaces, we define Lorentz-Herz spaces (Orlicz-Herz spaces) as a particular case of  $E(\mathcal{X})$, which are still generalizations of the classical Herz spaces. In this context (especially Lorentz-Herz spaces), we establish previously studied properties, demonstrate interpolation results, and prove the boundedness of important sublinear integral operators with kernels that satisfy a size condition.



\blfootnote{$^{*}$Corresponding Author} 
\blfootnote{G. Sankara Raju Kosuru: raju@iitrpr.ac.in} \blfootnote{%
Mohd Ashraf Bhat: ashraf74267@gmail.com} \blfootnote{%
Pawe{\l} Kolwicz: pawel.kolwicz@put.poznan.pl}

\vspace{1cm}


\section{Introduction}

\counterwithin{lemma}{section}

Over the years, many function spaces have been introduced in the literature to solve various contemporary problems arising in different areas of mathematics and allied subjects such as analysis, probability, and partial differential equations. For instance, K\"othe-Bochner spaces were introduced to provide a suitable framework to define integrals of Banach space-valued functions and to study their properties. The geometry of K\"othe-Bochner space has been intensively developed during the last decades (see \cite{Lin}). 

The more general construction of infinite direct sums (substitution spaces) has been employed to analyze a variety of problems in functional analysis, particularly in the geometry of Banach spaces and the theory of Banach lattices \cite{Ku-Lan, lau}, \cite[section 4]{Le-Piazza}).

 On the other hand, classical Herz spaces were initially introduced by C. S. Herz \cite{Herz} to provide a suitable setting for the range of the Fourier transform acting on a class of Lipschitz spaces. R. Johnson \cite{eqnorm} characterized the norm of this class of functions in terms of Lebesgue space norms over annuli in $\mathbb{R}^{N}$. The formulation of Herz spaces used in this paper first appeared (probably) in \cite{Garcia94} and \cite{Lu92}.
 Let $S_{k}=\{x\in \mathbb{R}^{N}:|x|<2^{k}\}$
and $\Omega _{k}=S_{k}\setminus S_{k-1}$ for $k\in \mathbb{Z}$. Denote $%
\tilde{\chi}_{\Omega _{k}}=\chi _{\Omega _{k}}$ for $k\in \mathbb{Z}^{+}$
and $\tilde{\chi}_{\Omega _{-1}}=\chi _{S_{-1}}$.

\begin{definition}
\addtocounter{lemmaaux}{1}

For $a\in \mathbb{R}$, $0<p,q\leq \infty $. The homogeneous Herz space $\dot{%
K}_{a,q}^{p}(\mathbb{R}^{N})$ and the non-homogeneous Herz space $K_{a,q}^{p}(\mathbb{R}^{N})$ are respectively defined by 
\begin{equation*}
\dot{K}_{a,q}^{p}(\mathbb{R}^{N})=\{f\in L_{loc}^{p}(\mathbb{R}^{N}\setminus
\{0\}):\Vert f\Vert _{\dot{K}_{a,q}^{p}(\mathbb{R}^{N})}<\infty \}
\end{equation*}%
\text{and}
\begin{equation*}
~ K_{a,q}^{p}(\mathbb{R}^{N})=\{f\in L_{loc}^{p}(\mathbb{R}^{N}):\Vert f\Vert
_{K_{a,q}^{p}(\mathbb{R}^{N})}<\infty \},
\end{equation*}%
where 
\begin{equation*}
\Vert f\Vert _{\dot{K}_{a,q}^{p}(\mathbb{R}^{N})}=\left( \sum_{k\in \mathbb{Z%
}}2^{kaq}\Vert f\chi _{\Omega _{k}}\Vert _{{L^{p}}}^{q}\right) ^{\frac{1}{q}%
} \mbox{ and } ~~
\Vert f\Vert _{K_{a,q}^{p}(\mathbb{R}^{N})}=\left( \sum_{k=-1}^{\infty
}2^{kaq}\Vert f\tilde{\chi}_{\Omega _{k}}\Vert _{{L^{p}}}^{q}\right) ^{\frac{%
1}{q}}.
\end{equation*}%
The usual modifications are made if $p$ and/or $q$ is infinite.
\end{definition}

In recent decades, there has been significant progress in the development of these spaces due to their numerous applications. For example, they are considered as good replacements for Hardy spaces when conducting the investigations on boundedness of non-translation invariant singular integral operators \cite{HTH1, HTH2}. Additionally, they are used in the characterization of multipliers on Hardy spaces \cite{Baernstein}, the regularity theory for elliptic equations in divergence form \cite{PDE}, the summability of Fourier transforms \cite{SFT}, and the Navier-Stokes equations \cite{NSE}. Moreover, Herz spaces have recently been used in the study of semilinear parabolic equations \cite{SLPE}.

Herz spaces can also be viewed as amalgam spaces having the
local component as the family $\mathcal{X=}\left( L^{p}\left( \Omega
_{k}\right) \right) _{k\in I}${\ (where }$\left( \Omega _{k}\right) _{k\in I}
${\ is the partition of }$\Omega $) and the global component as the weighted
sequence space $l_{q}\left( w_{a}\right)$  with weight $w_a=w_a\left( k\right) =2^{ak}$, equivalently, $K_{a,q}^{p}(%
\mathbb{R}^{N})=\left( \oplus _{k\in I}L^{p}\left( \Omega _{k}\right)
\right) _{l_{q}\left( w_{a}\right) }$. Moreover, different kinds of amalgam spaces have been considered and used by many authors (see, for example, \cite{Pesa}
for the ideal space approach and some history and references).

The main objective of this paper is to introduce and study amalgam-type spaces called K\"{o}the-Herz spaces $E\left( \mathcal{X}\right) $, {having the local component as the family }$\mathcal{X=}\left(
X\left( \Omega _{\alpha }\right) \right) _{\alpha \in I}$ of quasi-normed ideal spaces {\ (where }$\left(
\Omega _{\alpha }\right) _{\alpha \in I}${\ is the partition of }$\Omega $)
and the global component as the quasi-normed ideal sequence space $E$ over
the counting measure space $\left( I,2^{I},m\right)$. Recall that the theory of (quasi-)normed ideal spaces over some measure space $\left( \Omega ,\Sigma
,\mu \right) $ has been intensively and widely developed (see \cite{Sharpley1988, Kam-Mal-Per-2003, Kam-Mal-Isr, kantor, Kol2018-Posit, KLM-2019, Holder-lorentz, Lin, Lin-Tza-s, LinTza, Ma04, Ma08}). 
 K\"othe-Herz spaces will consolidate most of the above threads (among others) under a single umbrella and provide a clear picture of their similarities.


It is important to recognize that the K\"{o}the-Herz spaces encompass a wide range of classical function spaces found in the literature. However, investigating properties of such spaces can be a complex and parameter-dependent task, often requiring diverse methods. By unifying the essential features of various function spaces, we believe that the K\"{o}the-Herz spaces can simplify the study of these spaces and make the methods more easy and accessible. A glimpse of this will reflect within the paper, as we establish several significant results using short and elementary arguments.

Since we prove that the spaces $E\left( \mathcal{X}\right) $ are quasi-normed ideal spaces, it is pertinent to focus on properties that are fundamental 
from both lattice and function space perspectives. To illustrate the significance of K\"{o}the-Herz spaces, we focus on a specific example of the spaces $E\left( \mathcal{X}\right) $ by taking local component $X=L^{p,r}$- the Lorentz spaces\footnote{\label{fn} Lorentz spaces were introduced by G. G. Lorentz (\cite{Lorentz1, Lorentz2}) as a generalization of Lebesgue spaces, which are insufficient to describe certain properties of functions and operators. They have played a crucial role in analysis, particularly in interpolation theory and Fourier analysis (see \cite{Sharpley1988, Bergh, Caledron, Grafakos}).} and global component $E=l_{q}\left( w_{a}\right) $- the weighted Lebesgue sequence space. We call them the Lorentz-Herz spaces and they generalize the classical Herz spaces. Applying the known results about Lorentz spaces as well as our general results about the K\"{o}the-Herz spaces, we get the description of the basic properties of the Lorentz-Herz spaces. Furthermore, we study a real and complex interpolation
of Lorentz-Herz spaces. Finally, we show the boundedness of important
integral sublinear operators with a kernel.

The K\"othe-Herz spaces are a very general framework for the study of various function spaces. To better understand the significance of this framework, it is important to examine more examples. In this paper, we provide additional examples, including Orlicz-Herz spaces (see Appendix A). However, to keep the paper concise, we study of only the Lorentz-Herz spaces in detail.

The paper is structured as follows. In Section 2, we define the terminology used throughout the paper and provide some necessary background information. In Section 3, we establish the completeness of infinite direct sums, and present a criterion for order continuity. We then introduce K\"{o}the-Herz spaces $E\left(\mathcal{X}\right)$ and prove that they are quasi-normed ideal spaces. Moreover, we characterize their Fatou property, duality, and K\"{o}the duality, and derive the H\"{o}lder inequality as a corollary. We also provide a Banach function space characterization coming from \cite{Sharpley1988}.

In Subsection 3.1, we focus on a particular case of K\"{o}the-Herz spaces, namely the Lorentz-Herz spaces $HL_{a,q}^{p,r}$. By applying the general results obtained earlier, we show that these spaces are quasi-Banach ideal spaces with the Fatou property. We further investigate their order continuity, embeddings, and K\"{o}the duality, and demonstrate their stability under real and complex interpolation. Finally, we establish the boundedness of a class of operators on $HL_{a,q}^{p,r}$ by leveraging corresponding boundedness results on Lorentz spaces.

The paper concludes with a brief appendix (Appendix A) on Orlicz-Herz spaces.


\section{Preliminaries}

The functional $x\mapsto \Vert x\Vert $ on a given vector space $X$ is
called a \textit{quasi-norm} if the following three conditions are
satisfied: $\Vert x\Vert =0$ iff $x=0$; $\Vert ax\Vert =|a|\,\Vert x\Vert
,x\in X,a\in \mathbb{R}$ 
and there exists $C\geq 1$ such that $\Vert x+y\Vert \leq C(\Vert x\Vert +\Vert
y\Vert )$ for all $x,y\in X$. \ The smallest possible constant $C$ for which
the above inequality holds is called modulus of concavity and it is usually denoted by $C_{X}.$ A quasi-norm $\Vert \cdot \Vert $  is called a \textit{$p$-norm} (where $0<p\leq 1$) if, in addition, it is $p$-subadditive, that is, $\Vert x+y\Vert^{p}\leq \Vert x\Vert^{p}+\Vert y\Vert^{p}$ for all $x,y \in X$. Recall the \textit{Aoki--Rolewicz theorem} (cf. \cite[%
Theorem 1.3 on p. 7]{KPR84}, \cite[p. 86]{Ma04}, \cite[pp. 6--8]{Ma08}): if $\|\cdot\|$ is a quasi-norm on $X$, then there exists an equivalent \textit{$p$-norm} $\vertiii{\cdot}$ on $X$. More precisely, the equivalent \textit{$p$-norm} is given by 
\begin{equation*}
\vertiii{x}=\inf \left\{\left(\sum_{k=1}^{n}\Vert x_{k}\Vert_X ^{p} \right)^{1/p}\colon x=\sum_{k=1}^{n} x_{k}\right\},
\end{equation*}%
where the infimum is taken over all finite sequences $\{x_k\} \subset X$. The quasi-norm $\Vert \cdot \Vert $ induces
a metric topology on $X$. In fact, the metric can be defined by $d(x,y)=\vertiii
{x-y}^{p}$. We say that $X=(X,\Vert \cdot \Vert )$ is a \textit{%
quasi-Banach space} if it is complete for the aforementioned metric.
%
%
%
%

A triple $(\Omega ,\Sigma ,\mu )$ stands for a positive, complete and $%
\sigma $-finite measure space. Let $L^{0}=L^{0}(\Omega ,\Sigma ,\mu )$ be
the space of all (equivalence classes of) $\Sigma $-measurable functions $%
f:\Omega \rightarrow \mathbb{R}$. For every $f\in L^{0}$, we denote $supp(
f)=\overline{\{x\in \Omega :f(x)\neq 0\}}$. For any $f,g\in L^{0}$, we write $f\leq g$,
if $f(x)\leq g(x)$ almost everywhere with respect to the measure $\mu $ on
the set $\Omega $.
Set $L_{+}^{0}= \{f\in L^{0}:f\geq 0 \}$.
\begin{definition}
\label{ideal-space}A quasi-normed lattice [quasi-Banach lattice] $E=(E,\leq
,\Vert \cdot \Vert _{E})$ is called a \textit{quasi-normed ideal space} [%
\textit{quasi-Banach ideal space} or a \textit{quasi-K\"{o}the space}] if it
is a linear subspace of $L^{0}$ satisfying the following conditions:

\begin{itemize}
\item[(i)] If $f\in L^{0}$, $g\in E$ and $|f|\leq |g|$ $\mu $-a.e., then $%
f\in E$ and $\Vert f\Vert _{E}\leq \Vert g\Vert _{E}$.

\item[(ii)] There exists $f\in E$ which is strictly positive on the whole $%
\Omega $ (cf. \cite{LinTza}).
\end{itemize}

If additionally $\Vert \cdot \Vert _{E}$ is a norm, then $E$ is called
normed-\textit{ideal space [Banach ideal space].}
\end{definition}

For two ideal (quasi-)normed spaces $E$ and $F$ on $\Omega$ the symbol $E\overset{C}{%
\hookrightarrow }F$ means that the inclusion $E\subset F$ is continuous with
a norm which is not bigger than C, i.e., $\| x\| _{F} \leq C \| x\|_{E}$ for
all $x \in E$. In the case when the embedding $E\overset{C}{\hookrightarrow }%
F$ holds with some (unknown) constant $C>0$ we simply write $%
E\hookrightarrow F$. Moreover, $E = F$ (and $E\equiv F$) means that the
spaces are the same and the norms are equivalent (equal).

By a \textit{symmetric space} on $\Omega$ we mean a (quasi-)normed ideal
space $E=(E,\Vert \cdot \Vert _{E})$ with the additional property that for
any two equimeasurable functions $f\sim g,~f,g\in L^{0}(\Omega)$ (that is,
they have the same distribution functions $d_{f}=d_{g}$, where $%
d_{f}(\lambda )=\mu(\{x\in \Omega \colon |f(x)|>\lambda \}),\lambda \geq 0$)
and $f\in E$, we have $g\in E $ and $\Vert f\Vert _{E}=\Vert g\Vert _{E}$%
. In particular, $\Vert f\Vert _{E}=\Vert f^{\ast }\Vert _{E}$, where $%
f^{\ast }(t)=\mathrm{inf}\{\lambda >0\colon \ d_{f}(\lambda )\leq t\},\
t\geq 0$ is the decreasing (or non-increasing) rearrangement of $f$. The most common examples of symmetric spaces are provided by Orlicz spaces, Lorentz spaces and Marcinkiewicz spaces.

Recall that a quasi-normed ideal space $E$ has the \textit{Fatou property
(briefly, }$E\in \left( FP\right) $), if for any $f\in L^{0}$ and any $%
\left( f_{n}\right) _{n=1}^{\infty }$ in $E_{+}$ such that $f_{n}\uparrow
|f| $ $\mu $-a.e and $\sup_{n\in \mathbb{N}}\Vert f_{n}\Vert _{E}<\infty $, we get $%
f\in E$ and $\lim_{n}\Vert f_{n}\Vert _{E}=\Vert f\Vert _{E}$. If $E$ is a normed ideal space, then $E\in \left( FP\right) $ iff $E=(E')'$, where $E'$ is a K\"othe dual of $E$ (see \cite[p. 30]{LinTza}). 

Furthermore, a quasi-normed lattice $E$ is called \textit{order continuous} (%
$E\in (\mathrm{OC})$) if for each sequence $f_{n}\downarrow 0$, that is $%
f_{n}\geq f_{n+1}$ and $\inf_{n}f_{n}=0$, we have $\Vert f_{n}\Vert
_{E}\rightarrow 0$ (see \cite{kantor, LinTza, Wnuk}).
Moreover, $E\in \left( \mathrm{OC}\right) $ if and only if for every element 
$f\in E$ and each sequence $\left( f_{n}\right) $ in $E$ satisfying
conditions $\inf \left\{ f_{n},f_{m}\right\} =0$ for $n\neq m$ and $%
0\leq f_{n}\leq \left\vert f\right\vert $ we have $\left\Vert
f_{n}\right\Vert _{E}\rightarrow 0$ (see Theorem 2.1 in \cite{Kol2018-Posit}%
). If $E$ is a normed ideal space, then $E\in \left( OC\right) $ iff $E^{*}=E'$, where $E^{*}$ is a topological dual of $E$ (see \cite[p. 29]{LinTza}). 

In the paper we usually consider $\mathcal{X}=\left( X_{\alpha }\right) _{\alpha \in I}$
as the family of non-trivial quasi-normed ideal spaces. But sometimes we can
formulate our statements in more general form. Below we recall the notion of infinite direct sum.

\begin{definition}

Let $I$ be any countable set of indices and $E$ be a quasi-normed ideal
space over counting measure space $\left( I,2^{I},m\right) $. Given a family 
$\mathcal{X}=\left( X_{\alpha }\right) _{\alpha \in I}$ of non-trivial
quasi-normed spaces, the infinite direct sum $ \left( \oplus _{\alpha \in I}X_{\alpha}\right)
_{E} $ is defined by

\begin{equation*}
\left( \oplus _{\alpha \in I}X_{\alpha}\right)_{E} =\left\{ f=\left( f_{\alpha }\right) _{\alpha}:f_{\alpha }\in X_{\alpha }\text{ for each }\alpha \in I\text{ and }\left(
\left\Vert f_{\alpha }\right\Vert _{X_{\alpha }}\right) _{\alpha }\in
E\right\} .
\end{equation*}
\end{definition}

We equip this space with the functional $$\left\Vert
f\right\Vert _{\left( \oplus _{\alpha \in I}X_{\alpha}\right)
_{E} }=\left\Vert \left( \left\Vert
f_{\alpha }\right\Vert _{X_{\alpha }}\right) _{\alpha }\right\Vert _{E}.$$  
For simplicity, we will write 
\begin{equation*}
\mathcal{E}\left( \mathcal{X}\right) =\left( \oplus _{\alpha \in I}X_{\alpha}\right)
_{E}.
\end{equation*}

It is to be noted that the space $\mathcal{E}\left( \mathcal{X}\right) $ for a normed
case but also in more general contexts have been considered in many papers,
see for example \cite{Ku-Lan, lau, Le-Piazza}. Further, if $X_{\alpha}=X$ for each $\alpha$, then $\mathcal{E}\left( \mathcal{X}\right) $ is the classical K\"{o}the-Bochner sequence
space (see \cite{Lin}). For any $\alpha \in I$, let $e_{\alpha}=1_{\{\alpha%
\}}$ be the indicator function of the singleton set $\{\alpha\}$.

Denote by $C_{\alpha }$, the constant from quasi-triangle inequality for $X_{\alpha } $ and set
\begin{equation}
C=\sup_{\alpha }C_{\alpha }.
\end{equation}
The following facts in this section are trivial and well known but we present a proof of the first for reader's convenience.
\begin{proposition}
\label{quasi-normed} Let  $\mathcal{E}\left( \mathcal{X}\right) $ and $C$ be as above. Assume that $C<\infty $ and $D$ is the modulus of concavity for $E$.  Then

\begin{itemize}
\item[(i)] $\mathcal{E}\left( \mathcal{X}\right) $ is a linear space and the
functional $\| \cdot \| _{\mathcal{E}\left( \mathcal{X}\right) }$ is
a quasi-norm on $\mathcal{E}\left( \mathcal{X}\right) .$ If additionally $E$ and each $%
X_{\alpha }$ are normed spaces then the functional $\left\Vert \cdot
\right\Vert _{\mathcal{E}\left( \mathcal{X}\right) }$ is a norm on $\mathcal{E}\left( \mathcal{X}%
\right) .$

\item[(ii)] If $\mathcal{X}=\left( X_{\alpha }\right) _{\alpha \in I} $ is a
family of non-trivial quasi-normed lattices, then $( \mathcal{E}\left( \mathcal{X}\right) ,\| \cdot \|_{\mathcal{E}\left( \mathcal{X}\right) },\leq) $ is a quasi-normed lattice with the partial order defined by%
\begin{equation*}
f\leq g\text{ in } \mathcal{E}\left( \mathcal{X}\right) \text{ if and only if }~\forall ~\alpha ,\ f_{\alpha
}\leq g_{\alpha }\text{ in }X_{\alpha }.
\end{equation*}
\end{itemize}
\end{proposition}

\begin{proof}
It is enough to show that the functional $\| \cdot \|
_{\mathcal{E}\left( \mathcal{X}\right) }$ satisfies the quasi-triangle inequality. Let 
$f,g\in \mathcal{E}\left( \mathcal{X}\right) $, then $f=\left( f_{\alpha }\right)
_{\alpha }$ and $g=\left( g_{\alpha }\right) _{\alpha }$. Since each $%
X_{\alpha }$ is a quasi-normed space, therefore $\Vert f_{\alpha }+g_{\alpha
}\Vert _{X_{\alpha }}\leq C_{\alpha }\left( \Vert f_{\alpha }\Vert
_{X_{\alpha }}+\Vert g_{\alpha }\Vert _{X_{\alpha }}\right) $. Thus 
\begin{eqnarray*}
\Vert f+g\Vert _{\mathcal{E}(\mathcal{X})} &=&\left\Vert \left( \Vert f_{\alpha
}+g_{\alpha }\Vert _{X_{\alpha }}\right) _{\alpha }\right\Vert _{E} 
\leq \left\Vert \left( C_{\alpha }(\Vert f_{\alpha }\Vert _{X_{\alpha
}}+\Vert g_{\alpha }\Vert _{X_{\alpha }})\right) _{\alpha }\right\Vert _{E}
\\
&\leq &C\left\Vert \left( \Vert f_{\alpha }\Vert _{X_{\alpha }}\right)
_{\alpha }+\left( \Vert g_{\alpha }\Vert _{X_{\alpha }}\right) _{\alpha
}\right\Vert _{E} \\
&\leq &CD\left( \left\Vert \left( \Vert f_{\alpha }\Vert _{X_{\alpha
}}\right) _{\alpha }\right\Vert _{E}+\left\Vert \left( \Vert g_{\alpha
}\Vert _{X_{\alpha }}\right) _{\alpha }\right\Vert _{E}\right) \\
& \leq & CD\left( \Vert f\Vert _{\mathcal{E}(\mathcal{X})}+\Vert g\Vert _{\mathcal{E}(\mathcal{X}%
)}\right) .
\end{eqnarray*}%
\end{proof}

The example given below shows that the condition $\sup_{\alpha }C_{\alpha
}<\infty $ cannot be dropped.

\begin{example}

Let $A_{n}=[n,n+1),~n\in \mathbb{N}$ and $X_{n}=L^{p_{n}}[n,n+1)$ with the $%
p_{n}=1/n$ be equipped with quasi-norm 
\begin{equation*}
\Vert f\Vert _{X_{n}}=\left[ \int_{A_{n}}|f\left( t\right) |^{p_{n}}dt\right]
^{1/p_{n}}.
\end{equation*}%
Here $C_{n}=2^{n-1}$, thus $\sup_{n}C_{n}=+\infty $. Now let $E=\ell
^{q},1\leq q<\infty $ and choose $f,g\in \mathcal{E}(\mathcal{X})$ such that $%
f=(f_{n})_{n}=\left( \chi _{\lbrack n+\frac{1}{2},n+1 \rbrack} \right) _{n}$ and $%
g=(g_{n})_{n}=\left( \chi _{\lbrack n,n+\frac{1}{2})}\right) _{n}$. Then $%
\Vert f\Vert _{\mathcal{E}(\mathcal{X})},~\Vert g\Vert _{\mathcal{E}(\mathcal{%
X})}<\infty $, but $\Vert f+g\Vert _{\mathcal{E}(\mathcal{X})}=\infty $.
\end{example}

\begin{lemma}

\label{Isometries} $\left( i\right) $ The sequence space $E$ is
isometrically embedded in $\mathcal{E}(\mathcal{X})$.\newline
$\left( ii\right) $ For any $\alpha _{0}\in I$, $X_{\alpha _{0}}$ is
isometrically embedded in $\mathcal{E}(\mathcal{X})$.
\end{lemma}

\begin{proof} The proof is an easy exercise.
%
\end{proof}

\section{K\"{o}the-Herz spaces}
We begin this section by proving a few results that are valid for a more general infinite direct sum construction than the K\"othe-Herz spaces. Due to the facts discussed in the previous section, throughout the article we assume that
\begin{equation}
C=\sup_{\alpha }C_{\alpha } < \infty,
\end{equation}
where $C_{\alpha }$ is the constant from quasi-triangle inequality for $X_{\alpha } $. Note also that for the classical Herz spaces and for the Lorentz-Herz spaces (Section 3.1) this condition gets naturally satisfied.
The following theorem gives completeness of infinite direct sum $\mathcal{E}(\mathcal{X})$ by using the
completeness of individual components.

\begin{theorem}
\label{complete}Assume that $\mathcal{X}=\left( X_{\alpha }\right) _{\alpha
\in I}$ is a family of quasi-Banach spaces and $E$ is a quasi-Banach ideal
space over counting measure space $\left( I,2^{I},m\right) $. Then, the
space $\mathcal{E}(\mathcal{X})$ is complete.
\end{theorem} 

\begin{proof}
Suppose $f_{n}=\left( f_{\alpha }^{n}\right) _{\alpha },$ is a
sequence in $\mathcal{E}(\mathcal{X})$ such that 
\begin{equation*}
\sum\limits_{n=1}^{\infty }(CD)^{n}\Vert f_{n}\Vert _{\mathcal{E}(\mathcal{X})}<\infty,
\end{equation*}%
where $C$ and $D$ are as in Proposition \ref{quasi-normed}. Then given any $\beta \in I$, $\Vert f_{\beta }^{n}\Vert _{X_{\beta
}}e_{\beta }\leq \left( \Vert f_{\alpha }^{n}\Vert _{X_{\alpha }}\right)
_{\alpha }$ coordinate-wise for any $n\in \mathbb{N}$. Therefore, by ideal property of $E$,
we get $\Vert f_{\beta }^{n}\Vert _{X_{\beta }}\Vert e_{\beta }\Vert
_{E}\leq \Vert f_{n}\Vert _{\mathcal{E}(\mathcal{X})}$ for all $n\in \mathbb{N}$ and $%
\beta \in I$. Thus 
\begin{equation*}
\sum\limits_{n=1}^{\infty }(C_{\beta })^{n}\Vert f_{\beta }^{n}\Vert
_{X_{\beta }}\leq \frac{1}{\Vert e_{\beta }\Vert _{E}}\sum\limits_{n=1}^{%
\infty }(CD)^{n}\Vert f_{n}\Vert _{\mathcal{E}(\mathcal{X})}<\infty .
\end{equation*}%
Since $X_{\beta }$ is complete, by Theorem 1.1 from \cite{Ma04}, we have 
\begin{equation}
\sum\limits_{n=1}^{\infty }f_{\beta }^{n}\in X_{\beta }~\text{and}%
~\left\Vert \sum\limits_{n=1}^{\infty }f_{\beta }^{n}\right\Vert _{X_{\beta
}}\leq C_{\beta }\sum\limits_{n=1}^{\infty }C_{\beta }^{n}\Vert f_{\beta
}^{n}\Vert _{X_{\beta }}.  \label{eqn1}
\end{equation}%
Also, 
\begin{equation*}
\sum\limits_{n=1}^{\infty }(CD)^{n}\Vert f_{n}\Vert _{E(\mathcal{X}%
)}=\sum\limits_{n=1}^{\infty }D^{n}\left\Vert \left( \Vert C^{n}f_{\alpha
}^{n}\Vert _{X_{\alpha }}\right) _{\alpha }\right\Vert _{E}<\infty .
\end{equation*}%
As $E$ is complete, we have 
\begin{equation*}
\sum\limits_{n=1}^{\infty }\left( \Vert C^{n}f_{\alpha }^{n}\Vert
_{X_{\alpha }}\right) _{\alpha }\in E~\text{and}~\left\Vert
\sum\limits_{n=1}^{\infty }\left( \Vert C^{n}f_{\alpha }^{n}\Vert
_{X_{\alpha }}\right) _{\alpha }\right\Vert _{E}\leq
D\sum\limits_{n=1}^{\infty }D^{n}\left\Vert \left( \Vert C^{n}f_{\alpha
}^{n}\Vert _{X_{\alpha }}\right) _{\alpha }\right\Vert _{E}.
\end{equation*}%
Equivalently, 
\begin{equation}
\left( \sum\limits_{n=1}^{\infty }C^{n}\Vert f_{\alpha }^{n}\Vert
_{X_{\alpha }}\right) _{\alpha }\in E~\text{and}~\left\Vert \left(
\sum\limits_{n = 1}^{\infty} C^{n}\Vert f_{\alpha }^{n}\Vert _{X_{\alpha
}}\right) _{\alpha }\right\Vert _{E}\leq D\sum\limits_{n=1}^{\infty
}(CD)^{n}\left\Vert f_{n}\right\Vert _{\mathcal{E}(\mathcal{X})}.  \label{eqn2}
\end{equation}%
Since condition (\ref{eqn1}) is true for any $\beta \in I$, we obtain 
\begin{equation*}
\left( \left\Vert \sum\limits_{n=1}^{\infty }f_{\alpha }^{n}\right\Vert
_{X_{\alpha }}\right) _{\alpha }\leq C\left( \sum\limits_{n=1}^{\infty
}C^{n}\Vert f_{\alpha }^{n}\Vert _{X_{\alpha }}\right) _{\alpha }.
\end{equation*}%
It follows from equation (\ref{eqn2}) and ideal property of $E$ that 
\begin{equation}
\left( \left\Vert \sum\limits_{n=1}^{\infty }f_{\alpha }^{n}\right\Vert
_{X_{\alpha }}\right) _{\alpha }\in E~\text{and}~\left\Vert \left(
\left\Vert \sum\limits_{n=1}^{\infty }f_{\alpha }^{n}\right\Vert _{X_{\alpha
}}\right) _{\alpha }\right\Vert _{E}\leq C\left\Vert \left(
\sum\limits_{n=1}^{\infty }C^{n}\Vert f_{\alpha }^{n}\Vert _{X_{\alpha
}}\right) _{\alpha }\right\Vert _{E}.  \label{eqn3}
\end{equation}%
Conditions (\ref{eqn1}), (\ref{eqn2}) and (\ref{eqn3}) together imply 
\begin{equation*}
\sum\limits_{n=1}^{\infty }f_{n}\in \mathcal{E}(\mathcal{X})~\text{and}~\left\Vert
\sum\limits_{n=1}^{\infty }f_{n}\right\Vert _{\mathcal{E}(\mathcal{X})}\leq
CD\sum\limits_{n=1}^{\infty }(CD)^{n}\Vert f_{n}\Vert _{\mathcal{E}(\mathcal{X})}.
\end{equation*}%
Applying again Theorem 1.1 from \cite{Ma04}, we conclude that $\mathcal{E}(\mathcal{X})
$ is complete.

\end{proof}

\begin{theorem}
\label{oc}Let $\mathcal{X}=\left( X_{\alpha }\right) _{\alpha \in I}$ be a
family of quasi-normed lattices and $E$ be a quasi-normed ideal space over
counting measure space $\left( I,2^{I},m\right) .$ 
Then $\mathcal{E}\left( \mathcal{X}%
\right) \in \left( OC\right) $ if and only if $E\in \left( OC\right) $ and $%
X_{\alpha }\in \left( OC\right) $ for every $\alpha \in I.$
\end{theorem}

\begin{proof}

Suppose $E$ and $X_{\alpha }$ are order continuous for each $\alpha \in I$.
For any sequence $f_{n}=(f_{\alpha }^{n})_{\alpha }\in E(\mathcal{X}%
)$ with $f_{n}\downarrow 0$, 
 we have $\Vert
f_{\alpha }^{n}\Vert _{X_{\alpha }}\downarrow 0~$for every$~\alpha $.
Denoting $\widetilde{f_{n}}=\left( \Vert f_{\alpha }^{n}\Vert _{X_{\alpha
}}\right) _{\alpha }$,  we conclude that the sequence $\widetilde{f_{n}}$ is
pointwisely decreasing to zero in $E$.  Therefore, by $E\in (OC)$, we get $%
\left\Vert f_{n}\right\Vert _{\mathcal{E}\left( \mathcal{X}\right) }=\left\Vert 
\widetilde{f_{n}}\right\Vert _{E}\rightarrow 0$, i.e. $\mathcal{E}(\mathcal{X})\in (OC)
$. The converse is an outcome of the Lemma \ref{Isometries}. 
\end{proof}

Now we consider a particular case of the construction $%
\mathcal{E}\left( \mathcal{X}\right)$. Suppose $\mathcal{X}=\left( X_{\alpha }\right)
_{\alpha \in I}$ denotes a family of non-trivial quasi-normed ideal spaces
over the measure spaces $\left( \Omega _{\alpha },\Sigma _{\alpha },\mu _{\alpha }\right) $ chosen in a way described below$.$

\begin{definition} \label{Kothe-Herz} [\textbf{K\"{o}the-Herz spaces}]
Let $\left( \Omega _{\alpha }\right) _{\alpha \in I}$ be a measurable partition of $%
\Omega $, that is, $\Omega _{\alpha }\in \Sigma $ for $\alpha \in I$, $%
\Omega =\bigcup_{\alpha \in I}\Omega _{\alpha }~ \mu-a.e.$ and $\mu(\Omega _{\alpha }\cap
\Omega _{\beta })=0 $ for $\alpha ,\beta \in I,\alpha \neq \beta .$
Denote by $\Sigma _{\alpha }=\left\{ \Omega _{\alpha }\cap B : B\in \Sigma
\right\} $ and $\mu _{\alpha }:=\mu \big|_{\Sigma _{\alpha }} $. Let $X_{\alpha }$ be a
quasi-normed ideal space over measure space $\left( \Omega _{\alpha },\Sigma
_{\alpha },\mu _{\alpha }\right) ,\alpha \in I$ and $E$ be a quasi-normed ideal
space over counting measure space $\left( I,2^{I},m\right) $. Define 
\begin{equation*}
E\left( \mathcal{X}\right) =\left\{ f\in L^{0}(\Omega ,\Sigma ,\mu ):\ f\chi
_{\Omega _{\alpha }}\in X_{\alpha }\text{ for each }\alpha \in I\text{ and }%
\left( \left\Vert f\chi _{\Omega _{\alpha }}\right\Vert _{X_{\alpha
}}\right) _{\alpha }\in E\right\} 
\end{equation*}%
with the functional 
\begin{equation*}
\left\Vert f\right\Vert _{E\left( \mathcal{X}\right) }=\left\Vert \left(
\left\Vert f\chi _{\Omega _{\alpha }}\right\Vert _{X_{\alpha }}\right)
_{\alpha }\right\Vert _{E}.
\end{equation*}%
We will call $\left( E\left( \mathcal{X}
\right) ,\| \cdot \| _{E\left( \mathcal{X}\right) }\right) $ the \textit{K\"{o}the-Herz spaces}.
\end{definition}
  Clearly, the  K\"{o}the-Herz space  $E\left( \mathcal{X}\right)$ and the appropriate infinite direct sum $ \left( \oplus _{\alpha \in I}X_{\alpha}\right)
_{E} $ are linearly isometric. Thus we will identify these two objects.

In the following example, we assume that $a\in \mathbb{R},~$ and let $I=\mathbb{N}\cup \left\{ -1,0\right\} $ and $\left( \Omega _{k }\right) _{k \in I}$ be a partition of $\mathbb{R}^N $ given by
$\Omega _{-1}=\left\{ \omega \in \mathbb{R}^{N}:\left\vert \omega \right\vert
<1/2\right\} $ and
$\Omega _{k}=\left\{ \omega \in \Omega :2^{k-1}\leq \left\vert \omega
\right\vert <2^{k}\right\},$ $k \in \mathbb{N} \cup \{0\}$.
\begin{example} \label{example1} [\textbf{Herz spaces}] Taking 
$E=l_{q}\left( w_{a}\right) $ over $I$ with $w_{a}=w_{a}(k)=2^{ka}$ and $%
X_{k}=L^{p}\left( \Omega _{k}\right) $ for all $k\in I$ we get 
\begin{equation*}
E\left( \mathcal{X}\right) = K_{a,q}^{p}(\mathbb{R}^{N}).
\end{equation*}%
Of course, the spaces $K_{a,q}^{p}(\mathbb{R}^{N})$  and  $\left( \oplus _{k\in I}L^{p}\left( \Omega_{k}\right) \right) _{l_{q}\left( w_{a}\right) }$ are linearly isometric.
\end{example}
\begin{example} \label{example2} [\textbf{Other examples of Herz-type spaces}] 
Clearly, if we consider the family $\mathcal{X}=\left( X_{\alpha }\right)
_{\alpha \in I}$ of Orlicz, Lorentz or Marcinkiewicz spaces we can get the respective classes of Herz-type spaces (see the Lorentz-Herz spaces below). Also, we can fix in our construction $\mathcal{X}=\left( X_{\alpha }\right)
_{\alpha \in I}$ as the family of the Grand Lebesgue (Small Lebesgue) spaces which are still rearrangement invariant. Of course, we can go to the non-symmetric case taking the Musielak-Orlicz spaces (or as a particular case - Variable Lebesgue spaces, see \cite{NRZ}). In all cases above, the respective spaces $\left( E\left( \mathcal{X}\right)
,\| \cdot \| _{E\left( \mathcal{X}\right) }\right) $ have the ideal property (see the next proposition).

On the other hand, many authors consider the spaces of Herz type which are no longer ideal spaces (or even are not lattices). For example, D. Drihem studies the construction which, in our language, is just  $\left( \oplus _{k\in I}W^{p,m}\left( \Omega_{k}\right) \right) _{l_{q}\left( w_{a}\right) }$, where $W^{p,m}$ is the Sobolev space (see \cite[Def. 3.1]{Dr2}.
\end{example}
 Since there is a close resemblance between $E\left( \mathcal{X} \right)$ and the K\"othe-Bochner spaces, and our construction covers the classical Herz spaces, we refer to our construction as the K\"othe-Herz spaces. As Proposition \ref{quasi-normed} demonstrates, the K\"{o}the-Herz space $\left( E\left( \mathcal{X}%
\right) ,\| \cdot \| _{E\left( \mathcal{X}\right) }\right) $
is a quasi-normed lattice. However, the K\"{o}the-Herz spaces possess additional and useful structure - they are ideal spaces.

\begin{proposition}
\label{ideal} The K\"{o}the-Herz space $\left( E\left( \mathcal{X}\right)
,\| \cdot \| _{E\left( \mathcal{X}\right) }\right) $ is a
quasi-normed ideal space over $(\Omega,\Sigma ,\mu ).$
\end{proposition}

\begin{proof}
In view of Proposition \ref{quasi-normed}, we only need to show that the space $\left( E\left( \mathcal{X}\right) ,\| \cdot \| _{E\left( \mathcal{X}\right) }\right) $ has ideal property. Let $f,g \in L^0$ such that $|f| \leq |g|~\mu-a.e.$ and $g \in E\left(\mathcal{X}\right)$. Then, for any $\alpha \in I$, $|f|\chi _{\Omega _{\alpha}} \leq |g|\chi _{\Omega _{\alpha}}~\mu-a.e.$ and $g\chi_{\Omega_{\alpha}} \in X_{\alpha}$. Therefore, by ideal property of $X_{\alpha}$, $f\chi_{\Omega_{\alpha}} \in X_{\alpha}$ and $\left\Vert f\chi_{\Omega _{\alpha}}\right\Vert _{X_{\alpha}} \leq \left\Vert g\chi_{\Omega _{\alpha}}\right\Vert _{X_{\alpha}}$. Also, since $\left( \left\Vert f\chi_{\Omega _{\alpha}}\right\Vert _{X_{\alpha}}\right) _{\alpha} \leq \left( \left\Vert g\chi_{\Omega _{\alpha}}\right\Vert _{X_{\alpha}}\right) _{\alpha}$ and $\left( \left\Vert g\chi_{\Omega _{\alpha}}\right\Vert _{X_{\alpha}}\right) _{\alpha} \in E$, by ideal property of $E$, we conclude that $\left( \left\Vert f\chi_{\Omega _{\alpha}}\right\Vert _{X_{\alpha}}\right) _{\alpha} \in E$ and $\left\Vert \left(\left\Vert f\chi _{\Omega _{\alpha}}\right\Vert _{X_{\alpha}}\right) _{\alpha}\right\Vert_{E} \leq \left\Vert \left(\left\Vert g\chi _{\Omega _{\alpha}}\right\Vert_{X_{\alpha}}\right) _{\alpha}\right\Vert_{E}$ i.e. $f \in E\left(\mathcal{X} \right)$ and
$\left\Vert f \right\Vert _{E\left( \mathcal{X}\right) }  \leq \left\Vert g \right\Vert _{E\left( \mathcal{X}\right) } $.
\end{proof}

Applying Theorem \ref{complete} we immediately conclude
\begin{corollary}
 \label{Herz-complete}Assume that $\mathcal{X}=\left( X_{\alpha }\right) _{\alpha
\in I}$ is a family of quasi-Banach ideal spaces and $E$ is a quasi-Banach ideal
space over counting measure space $\left( I,2^{I},m\right) $. Then, the K\"{o}the-Herz
space ${E}(\mathcal{X})$ is complete.   
\end{corollary}

For an ideal space $E$ we denote by $E_{s}$ the set of simple functions in $E$ having the finite measure. Applying Lemma 3 from \cite[p. 98]{kantor} (which is still true for quasi-normed space with the same proof) and Theorem \ref{oc}, we get

\begin{lemma}\label{dens} [\textbf{Density of the set of simple functions}] 
    Suppose $\mathcal{X}=\left( X_{\alpha }\right) _{\alpha
\in I}$ is a family of quasi-normed ideal spaces and $E$ is a quasi-normed ideal
space over counting measure space $\left( I,2^{I},m\right) $. If $E\in \left( OC\right) $ and $X_{\alpha }\in \left( OC\right) $ for every $\alpha \in I$, then the set $(E(\mathcal{X}))_{s}$ is dense in $(E(\mathcal{X}))$.

\end{lemma}

Now we apply Theorem 3 from \cite[p. 98]{kantor}. A careful reading of the proof details of of the all necessary facts in \cite{kantor} shows that it is still true for a quasi-Banach ideal space (see for example Lemma 4.1 in \cite{Fo-Hu-Ko} which is a generalization of Theorem 1 in \cite[p. 96]{kantor}, also we apply Proposition 2.2 in \cite{lee}). In consequence, using also Theorem \ref{oc}, we obtain

\begin{lemma}\label{sep} [\textbf{Separability}] 
   Suppose $\mathcal{X}=\left( X_{\alpha }\right) _{\alpha
\in I}$ is a family of quasi-Banach ideal spaces and $E$ is a quasi-Banach ideal
space over counting measure space $\left( I,2^{I},m\right) $. Then the K\"{o}the-Herz
space ${E}(\mathcal{X})$ is separable if and only if $\mu$ is separable, $E\in \left( OC\right) $ and $X_{\alpha }\in \left( OC\right) $ for every $\alpha \in I$.
\end{lemma}

\begin{theorem} \label{FP} [\textbf{Fatou property}]
 The K\"{o}the-Herz space $E\left( \mathcal{X}\right) $ has the
Fatou property if and only if both $E$ and $X_{\alpha }$ ($\alpha \in I$%
) have the Fatou property.
\end{theorem}

\begin{proof}
In the light of Lemma \ref{Isometries}, we only need to prove the sufficiency. Suppose that each $X_{\alpha}, ~\alpha \in I$ and $E$ have the Fatou property. Let $f_n \in E\left(\mathcal{X} \right)$ and $f \in L^0$ such that $f_n \uparrow f$ and $\sup_n\left\Vert f_n \right\Vert _{E\left( \mathcal{X}\right) }< \infty$. Then $\forall~\alpha \in I$, $f_n \chi_{\Omega_{\alpha}} \in X_{\alpha}$, $f_n \chi_{\Omega_{\alpha}} \uparrow f \chi_{\Omega_{\alpha}} \in L^0$ and $\sup_n \|f_n \chi_{\Omega_{\alpha}}\|_{X_{\alpha}}< \infty$. Using Fatou's property of $X_{\alpha}$, we have $f \chi_{\Omega_{\alpha}} \in X_{\alpha}$ and  
$\|f_n \chi_{\Omega_{\alpha}}\|_{X_{\alpha}} \uparrow \|f \chi_{\Omega_{\alpha}}\|_{X_{\alpha}}$. Set $\tilde{f_n}=\left(\|f_n \chi_{\Omega_{\alpha}}\|_{X_{\alpha}} \right)_{\alpha}$ and $\tilde{f}=\left(\|f \chi_{\Omega_{\alpha}}\|_{X_{\alpha}} \right)_{\alpha}$, then $\tilde{f_n} \uparrow \tilde{f}$ pointwise, $\tilde{f_n} \in E$ and $\sup_n\|\tilde{f_n}\|_E < \infty$. Therefore, by Fatou's property of $E$, $\tilde{f} \in E$ and $\|\tilde{f_n}\|_E \uparrow \|\tilde{f}\|_E$. Equivalently, $f \in E\left( \mathcal{X}\right)$ and $\left\Vert f_n \right\Vert _{E\left( \mathcal{X}\right) }  \uparrow \left\Vert f \right\Vert _{E\left( \mathcal{X}\right) } $.
\end{proof}

It is known that a quasi-normed ideal space with the Fatou property is
complete (see \cite[Lemma 2.1]{Kam-Mal-Per-2003}, cf. \cite[Lemma 4.2]{Fo-Hu-Ko}). Consequently, we get a
way of verification for the completeness (cf. Theorem \ref{complete}).

\begin{corollary}
\label{KH-complete} If $E\in \left( FP\right) $ and $X_{\alpha }\in \left(
FP\right) $ for every $\alpha \in I,$ then the K\"{o}the-Herz space $\left(
E\left( \mathcal{X}\right) ,\| \cdot \| _{E\left( \mathcal{X%
}\right) }\right) $ is a quasi-Banach ideal space.
\end{corollary}

\begin{lemma}
\label{inlusions} Let $E$ and $F$ be two quasi-normed ideal spaces over
counting measure space $\left( I,2^{I},m\right) $. Assume that $\mathcal{X}%
=\left( X_{\alpha }\right) _{\alpha \in I}$ and $\mathcal{Y}=\left(
Y_{\alpha }\right) _{\alpha \in I}$ are suitable families of quasi-Banach
spaces over measure space $\left( \Omega _{\alpha },\Sigma _{\alpha },\mu
_{\alpha }\right) ,\alpha \in I$ as in Definition \ref{Kothe-Herz}. The
following implications hold:

\begin{enumerate}
\item If $E\hookrightarrow F$ then $E(\mathcal{X})\hookrightarrow F(\mathcal{%
X})$.

\item If $X_{\alpha }\overset{K_{\alpha}}{\hookrightarrow} Y_{\alpha }~\forall ~\alpha \in I$ 
 and $\sup_{\alpha }K_{\alpha }=K<\infty $, then $E(%
\mathcal{X})\overset{K}{\hookrightarrow} E(\mathcal{Y})$.
\end{enumerate}
\end{lemma}

The proof this lemma is immediate and can be omitted for brevity. Recall only that for quasi-Banach ideal spaces the
inclusion is always continuous, which follows from the closed graph theorem
which is still true for quasi-Banach spaces (see \cite{KPR84}).

The fundamental tool in theory of ideal spaces is the H\"{o}lder
inequality which needs the notion of the K\"{o}the dual (see \cite{Lin}, \cite{LinTza} for definition and details).

\begin{theorem} \label{dualKHS} [\textbf{The dual of infinite direct sums $E\left( \mathcal{X%
}\right) $}]
Let $\mathcal{X}=\left( X_{\alpha }\right) _{\alpha \in I}$ be a family of
Banach spaces and $E$ be a Banach ideal space over counting measure space $%
\left( I,2^{I},m\right) $. Suppose $\mathcal{X}^{\ast }=\left( X_{\alpha
}^{\ast }\right) _{\alpha \in I}$.
\begin{enumerate}
\item If $E\in \left( OC\right) $ and $E^{\ast }\in \left(
OC\right) $, then 
$
\left( E\left( \mathcal{X}\right) ,\| \cdot \| _{E\left( 
\mathcal{X}\right) }\right) ^{\ast }\equiv \left( E^{\ast }\left( \mathcal{X}%
^{\ast }\right) ,\| \cdot \| _{E^{\ast }\left( \mathcal{X}%
^{\ast }\right) }\right) .$
\item If $E$ and $X_{\alpha },$ for each $\alpha ,$ are
reflexive, then $E\left( \mathcal{X}\right) $ is.
\end{enumerate}
\end{theorem}

\begin{proof}

Note only that if $E$ is order continuous Banach ideal space over counting
measure space $\left( I,2^{I},m\right) $, then the sequence $\left(
e_{\alpha }\right) _{\alpha \in I}$ of unit vectors forms a $1$%
-unconditional basis in $E$ (see \cite[Proposition 1.c.6, p.18]{Lin-Tza-s}).
Moreover, if additionally $E^{\ast }\in \left( OC\right) $ then $\left(
e_{\alpha }\right) _{\alpha \in I}$ is a shrinking basis, because the
biorthogonal functionals $\left( e_{\alpha }^{\ast }\right) _{\alpha \in I}$
forms a basis in $E^{\ast }.$ Then it is enough to apply Proposition 4.8
from \cite{lau}. 
\end{proof}

\begin{remark}
It is worth noting that if $E$ is reflexive, then $E$ and $E^*$ both belong to the class $(OC)$ (see \cite[Theorem 8.1]{Wnuk}). However, the converse is not always true. For instance, consider the space $E=c_{0}$.
\end{remark}

Here and hereafter, for any space $X$ under consideration, we use the notation $X^{'}$ to represent its K\"{o}the dual. Recall again that $X\in \left( OC\right) $ if and only if $X^{\ast }\equiv
X^{^{\prime }}$ (see \cite{Lin, LinTza}). Thus, applying Theorem \ref{oc}, we
conclude immediately.

\begin{corollary}
\label{Kothe-dual} [\textbf{The K\"{o}the dual of K\"{o}the-Herz spaces}] Let $\mathcal{X}=\left( X_{\alpha }\right) _{\alpha \in I}$ be a family of
Banach ideal spaces (as in Definition \ref{Kothe-Herz}) and $E$ be a Banach
ideal space over counting measure space $\left( I,2^{I},m\right) $. Denote $%
\mathcal{X}^{\ast }=\left( X_{\alpha }^{\ast }\right) _{\alpha \in I}.$ Let $E\left( \mathcal{X}\right) $ be the appropriate K\"{o}the-Herz space.
\newline 
If $E\in \left( OC\right) $ and $E^{\ast }\in \left( OC\right) $ then 
\begin{equation*}
\left( E\left( \mathcal{X}\right) ,\| \cdot \| _{E\left( 
\mathcal{X}\right) }\right) ^{\ast }\equiv \left( E^{^{\prime }}\left( 
\mathcal{X}^{\ast }\right) ,\| \cdot \| _{E^{^{\prime
}}\left( \mathcal{X}^{\ast }\right) }\right) .
\end{equation*}%
If additionally $X_{\alpha }\in \left( OC\right) $ for each $\alpha ,$ then%
\begin{equation*}
\left( E\left( \mathcal{X}\right) ,\| \cdot \| _{E\left( 
\mathcal{X}\right) }\right) ^{^{\prime }}\equiv \left( E^{^{\prime }}\left( 
\mathcal{X}^{^{\prime }}\right) ,\| \cdot \| _{E^{^{\prime
}}\left( \mathcal{X}^{^{\prime }}\right) }\right) ,
\end{equation*}%
where $\mathcal{X}^{^{\prime }}=\left( X_{\alpha }^{^{\prime }}\right)
_{\alpha \in I}.$
\end{corollary}


Applying the definition of K\"{o}the dual (\cite{Lin}) and Corollary %
\ref{Kothe-dual}, we get the following H\"{o}lder type inequality.

\begin{proposition}
\label{HolderIq} [\textbf{H\"{o}lder inequality}]
Let $\mathcal{X}=\left( X_{\alpha }\right) _{\alpha \in I}$  and $E$ be as in Corollary \ref{Kothe-dual}. Suppose $E\in \left( OC\right) $, $E^{\ast
}\in \left( OC\right) $ and $X_{\alpha }\in \left( OC\right) $ for each $%
\alpha .$ Denote $\mathcal{X}^{^{\prime }}=\left( X_{\alpha }^{^{\prime
}}\right) _{\alpha \in I}.$ For all $f\in E\left( \mathcal{X}\right) $ and $%
g\in E^{^{\prime }}\left( \mathcal{X}^{^{\prime }}\right) $, we have $fg\in
L^{1}(\Omega ,\Sigma ,\mu )$ and 
\begin{equation*}
\int_{\Omega} \left\vert fg\right\vert d\mu \leq \left\Vert f\right\Vert _{E\left( 
\mathcal{X}\right) }\cdot \left\Vert g\right\Vert _{E^{^{\prime }}\left( 
\mathcal{X}^{^{\prime }}\right) }.
\end{equation*}
\end{proposition}

 Recall also the definition of a \textit{function norm} . 
\begin{definition}
\label{function norm} $($\cite[Definition 1.1, p.2]{Sharpley1988}$)$ \label{BFS} A
map $\eta :L_{+}^{0}(\Omega ,\Sigma ,\mu )\rightarrow \lbrack 0,\infty ]$ is
said to be a (Banach) function norm over $\Omega $, if the following
properties hold for all $f,g,f_{n}~(n=1,2,3,\cdots )$ in $L_{+}^{0}$ and for
all $\mu $-measurable subsets $E\subseteq \Omega $: \newline
${(P1)}~\eta (f)=0\iff f=0~\mu -a.e.$; ~ $\eta (\alpha f)=\alpha \eta
(f)~~(\alpha \geq 0)$; ~ $\eta (f+g)\leq \eta (f)+\eta (g)$. \newline
${(P2)}$ If $0\leq f\leq g~~\mu -a.e.$, then $\eta (f)\leq \eta (g)$. 
\newline
${(P3)}$ If $0\leq f_{n}\uparrow f~\mu -a.e.$, then $\eta (f_{n})\uparrow
\eta (f)$. \newline
${(P4)}$ If $\mu (E)<\infty $, then $\eta (\chi _{E})<\infty $. \newline
${(P5)}$ If $\mu (E)<\infty $, then $\int_{E}fd\mu \leq C_{E}\eta (f)$ for
some finite positive constant $C_{E}$, independent of $f$. \newline
If $\eta $ is a function norm, then a \textit{Banach function space} $X=(X,\eta)$ is
the collection $X=\{f\in L^{0}(\Omega ,\Sigma ,\mu ):\eta (|f|)<\infty \}$.
\end{definition}

\begin{theorem}
\label{BFS1} Suppose that $\left\{ X_{\alpha }\right\} _{\alpha \in I}$ is a
family Banach ideal spaces in the sense of Definition \ref{Kothe-Herz} and $%
E $ is a Banach ideal space over counting measure space $\left(
I,2^{I},m\right) $. Assume that $E\in \left( FP\right) $ and $X_{\alpha }\in
\left( FP\right) $ for each $\alpha $. Also suppose $E\in \left( OC\right) $%
, $E^{\ast }\in \left( OC\right) $ and $X_{\alpha }\in \left( OC\right) $
for each $\alpha .$ Then the K\"{o}the-Herz space $\left( E\left( \mathcal{X}%
\right) ,\| \cdot \| _{E\left( \mathcal{X}\right) }\right) $
is a Banach function space if and only if for every $A\in \Sigma $ with $\mu
\left( A\right) <\infty $, the following conditions hold:

\begin{itemize}
\item[(a)] $\left \Vert \left( \left\Vert \chi _{\Omega _{\alpha}\cap
A}\right\Vert _{X_{\alpha}}\right) _{\alpha} \right\Vert_E< \infty.$
\item[(b)] $\left\Vert \left( \left\Vert \chi _{\Omega _{\alpha}\cap
A}\right\Vert _{X_{\alpha}^{^{\prime }}}\right)
_{\alpha}\right\Vert_{E^{^{\prime }}}< \infty,$
\end{itemize}

where $X_{\alpha}^{^{\prime }}$ and $E^{^{\prime }}$ are K\"othe duals of $%
X_{\alpha}$ and $E$, respectively.
\end{theorem}

\begin{proof}

The sufficiency: The properties $\left( P1\right) ,~\left( P2\right) $ and $%
\left( P3\right) $ of Definition \ref{BFS} follow from Proposition \ref%
{quasi-normed}, Proposition \ref{ideal} and Theorem \ref{FP} respectively.
Let $\mu \left( A\right) <\infty $. Then, by condition $(a)$ above 
\begin{equation*}
\Vert \chi _{A}\Vert _{E(\mathcal{X})}=\left\Vert \left( \left\Vert \chi
_{\Omega _{\alpha }\cap A}\right\Vert _{X_{\alpha }}\right) _{\alpha
}\right\Vert _{E}<\infty ,
\end{equation*}%
which confirms the property $(P4)$. Finally, using condition $(b)$ and
Proposition \ref{HolderIq}, we get 
\begin{eqnarray*}
\int_{A}fd\mu  &=&\int_{\Omega }f\chi _{A}d\mu  \leq \left\Vert f\right\Vert _{E\left( \mathcal{X}\right) }\cdot
\left\Vert \chi _{A}\right\Vert _{E^{^{\prime }}\left( \mathcal{X}^{^{\prime
}}\right) } \\
&=&\left\Vert f\right\Vert _{E\left( \mathcal{X}\right) }\cdot \left\Vert
\left( \left\Vert \chi _{\Omega _{\alpha }\cap A}\right\Vert _{X_{\alpha
}^{^{\prime }}}\right) _{\alpha }\right\Vert _{E^{^{\prime }}} 
=C_{E(\mathcal{X})}\cdot \left\Vert f\right\Vert _{E\left( \mathcal{X}%
\right) },
\end{eqnarray*}%
where $C_{E(\mathcal{X})}=\left\Vert \left( \left\Vert \chi _{\Omega
_{\alpha }\cap A}\right\Vert _{X_{\alpha }^{^{\prime }}}\right) _{\alpha
}\right\Vert _{E^{^{\prime }}}$. This gives the property $\left(
P5\right) $.

The necessity: Assume that $\left( E\left( \mathcal{X}\right)
,\| \cdot \| _{E\left( \mathcal{X}\right) }\right) $ is a
Banach function space in the sense of Definition \ref{BFS}, then so is its K\"othe dual (\cite[Theorem 2.2, p. 8]{Sharpley1988}). Therefore, by property 
$(P4)$ of $E(\mathcal{X})$ and $E^{^{\prime }}(%
\mathcal{X}^{^{\prime }})$, we have $\left\Vert \left( \left\Vert \chi _{\Omega
_{\alpha }\cap A}\right\Vert _{X_{\alpha }}\right) _{\alpha }\right\Vert
_{E}<\infty $ and $\left\Vert \left( \left\Vert \chi _{\Omega _{\alpha }\cap
A}\right\Vert _{X_{\alpha }^{^{\prime }}}\right) _{\alpha }\right\Vert
_{E^{^{\prime }}}<\infty $.\newline
\end{proof}

\subsection{Lorentz-Herz spaces}
\counterwithin{lemma}{subsection}

Now we turn our attention to a special case of K\"{o}the-Herz spaces, which retains the essential features of the classical Herz spaces $K_{a,q}^{p}(\mathbb{R}^{N})$ while allowing for greater generality. Recall that the Lorentz spaces $L^{p,r}$ are usually considered with the
following two equivalent norms (see \cite[p. 216]{Sharpley1988}).

\begin{definition}
\addtocounter{lemmaaux}{1}
\label{deflorentzspace} The Lorentz space $L^{p,r}=L^{p,r}(\Omega)$ consists
of all $\mu$-measurable functions on $\Omega$ for which the functional $%
\|f\|_{p,r}$ is finite, where 
\begin{equation*}
\left\Vert f\right\Vert _{p,r}=\left\{ 
\begin{array}{ccc}
\left( \int_{0}^{\infty }\left( t^{\frac{1}{p}}f^{\ast }\left( t\right)
\right) ^{r}\frac{dt}{t} \right)^{1/r} & \text{if} & 0<p<\infty ,0<r<\infty
\\ 
\sup_{t>0}t^{\frac{1}{p}}f^{\ast }\left( t\right) & \text{if} & 0<p\leq
\infty ,r=\infty,%
\end{array}%
\right.
\end{equation*}
\end{definition}
\noindent and $f^{*}$ is decreasing (non-increasing) rearrangement of $f$. Obviously, $L^{p,p}\equiv L^{p}$. It should be noted that $\|\cdot\|_{{p,r}}$ is not necessarily a norm but it
is a quasi-norm (see \cite[p. 216]{Sharpley1988}). However, we can equip $%
L^{p,r}$ with the functional $\|\cdot\|_{{(p,r)}}$ given by 
\begin{equation*}
\left\Vert f\right\Vert _{\left( p,r\right) }=\left\{ 
\begin{array}{ccc}
\left( \int_{0}^{\infty }\left( t^{\frac{1}{p}}f^{\ast \ast }\left( t\right)
\right) ^{r}\frac{dt}{t} \right)^{1/r} & \text{if} & 0<p<\infty ,0<r<\infty
\\ 
\sup_{t>0}t^{\frac{1}{p}}f^{\ast \ast }\left( t\right) & \text{if} & 0<p\leq
\infty ,r=\infty,%
\end{array}%
\right.
\end{equation*}
where $\displaystyle f^{**}(t)=\frac{1}{t}\int_0^t f^{*}(t) dt$ is the maximal average function. Then $\left(
L^{p,r},\|\cdot\|_{(p,r)} \right)$ is a normed space for $0<p< \infty$, $%
1 \leq r \leq \infty$ or $p=r=\infty$.

Clearly, if we admit above that $p=\infty$ for $0<r<\infty$ then $%
L^{\infty,r}=\{ 0 \}$ for each $0<r<\infty$. Thus we assume always that if $%
p=\infty$ then $r=\infty$. We borrow the the following notation from \cite{Sharpley1988}: $L^{p,r}=\left( L^{p,r},\| \cdot \|
_{p,r}\right) $ and $L^{\left( p,r\right) }=\left( L^{p,r},\left\Vert \cdot
\right\Vert _{\left( p,r\right) }\right).$ Obviously, since $f^{*} \leq
f^{**}$, $L^{(p,r)} \hookrightarrow L^{p,r} $ with 
\begin{equation*}
\|f\|_{L^{p,r}} \leq \|f\|_{L^{(p,r)}}.
\end{equation*}
Moreover (see \cite{Sharpley1988}), if $1<p \leq \infty$ and $1 \leq r \leq
\infty$ then 
\begin{equation}\label{L-relation}
\|f\|_{L^{(p,r)}} \leq \frac{p}{p-1}\|f\|_{L^{p,r}}.
\end{equation}
Of course, the spaces $\left( L^{p,r}, \|\cdot\|_{{p,r}} \right)$ and $%
\left( L^{p,r}, \|\cdot\|_{{(p,r)}} \right)$ are symmetric.

Here and Hereafter we consider ideal spaces over the Lebesgue measure space $(\Omega, \Sigma, \mu)$, where $\Omega \subset \mathbb{R}^{N}$.

\begin{definition} \label{DefLHS} [\textbf{The Lorentz-Herz spaces}] Let $\mu$ be Lebesgue measure,  $a\in \mathbb{R}$, $\Omega \subset \mathbb{R}^{N},$ $I=%
\mathbb{N}\cup \left\{ -1,0\right\} , $ $\Omega _{-1}=\left\{ \omega \in
\Omega :\left\vert \omega \right\vert <1/2\right\} $ and
$\Omega _{k}=\left\{ \omega \in \Omega :2^{k-1}\leq \left\vert \omega
\right\vert <2^{k}\right\} $ for $k=0,1,2, \cdots$ . We define the Lorentz-Herz spaces $HL_{a,q}^{ p,r }$ as the respective K\"{o}the-Herz spaces
\begin{equation}
HL_{a,q}^{ p,r }=  E\left( \mathcal{X}\right),  
\end{equation} 
where $E=l_{q}\left( w_{a}\right) $ over $I$ with $w_a=w_a\left( k\right) =2^{ak}$
and $X_{k}=L^{p,r}\left( \Omega _{k}\right) $ for all $k\in I$.  Similarly, we
define the spaces 
\begin{equation}
HL_{a,q}^{\left( p,r\right) }=  E\left( \mathcal{X}\right),  
\end{equation} 
where $E=l_{q}\left( w_{a}\right) $ over $I$ with $w_a=w_a\left( k\right) =2^{ak}$
and $X_{k}=L^{(p,r)}\left( \Omega _{k}\right) $ for all $k\in I$. 

\end{definition}
When $1 < p \leq \infty$ and $1 \leq q,r \leq \infty$, it is not difficult to observe that $HL_{a,q}^{\left( p,r\right) }$ defines a norm. Moreover, using equation (\ref{L-relation}) in this case, we can establish the following inequalities: 
$$\|f\|_{HL_{a,q}^{p,r}} \leq  \|f\|_{HL_{a,q}^{\left( p,r\right)}} \leq \frac{p}{p-1}\|f\|_{HL_{a,q}^{ p,r}}.$$

Furthermore, it is well known that (see \cite{Krist})

\begin{enumerate}
\item The functional $\| \cdot \| _{p,r}$ in $L^{p,r}$ is a
quasi-norm if either $0<p,r<\infty $ or $r=\infty $ and $0<p\leq \infty $.
\item The functional $\| \cdot \| _{p,r}$ in $L^{p,r}$ is a
norm if and only if either $1\leq r\leq p<\infty $ or $p=r=\infty $.
\item The functional $\| \cdot \| _{\left( p,r\right) }$ in 
$L^{p,r}$ is a quasi-norm for all $0<r\leq \infty $ and $0<p<\infty $.
\item The functional $\| \cdot \| _{\left( p,r\right) }$ in 
$L^{p,r}$ is a norm if either $1\leq r\leq \infty $ and $1<p<\infty $ or $%
p=r=\infty $.
\end{enumerate}

In all cases 1-4 above $L^{p,r}\in \left( FP\right)$ and $L^{(p,r)}\in
\left( FP\right)$, whence they are all complete. Clearly $E=l_{q}\left(
w_{a}\right)$ over $I$ with $w\left( k\right) =2^{ak}$ ($a\in \mathbb{R}$)
is also a quasi-Banach sequence space with the Fatou property for $0<q \leq \infty$. Applying Proposition \ref{quasi-normed}, Proposition \ref{ideal} and Corollary %
\ref{KH-complete}, we get immediately

\begin{corollary}
\begin{itemize}
\item[(i)] Suppose $0<q\leq \infty .$ If either $0<p,r<\infty $ or $r=\infty 
$ and $0<p\leq \infty, $ \ then $HL_{a,q}^{p,r}$ is a quasi-Banach ideal
space with the Fatou property.
\item[(ii)] Assume that $1\leq q\leq \infty .$ If either $1\leq r\leq
p<\infty $ or $p=r=\infty, $\ then $HL_{a,q}^{p,r}$ is a Banach ideal space
with the Fatou property.

\item[(iii)] If $0<q\leq \infty ,$ $0<r\leq \infty $ and $0<p<\infty, $\ then 
$HL_{a,q}^{\left( p,r\right) }$ is a quasi-Banach ideal space with the Fatou
property.
\item[(iv)] Suppose $1\leq q\leq \infty .$ If either $1\leq r\leq \infty $
and $1<p<\infty $\ or $p=r=\infty, $ then $HL_{a,q}^{\left( p,r\right) }$ is
a Banach ideal space with the Fatou property.
\end{itemize}
\end{corollary}

\begin{theorem}\label{PK}
\label{oc-for-lor-l_q} Let $0<p<\infty$ and $r>0$. The Lorentz spaces $%
\left( L^{p,r},\Vert \cdot \Vert _{{p,r}}\right) $ and $\left(
L^{p,r},\Vert \cdot \Vert _{{(p,r)}}\right) $ are order continuous if and
only if $r<\infty $. 
Similarly, for $E=l_{q}\left( w_{a}\right) $, $E\in (OC)$ if and only if $%
0<q<\infty $. 
\end{theorem}

\begin{proof}
The claim for $E=l_{q}\left( w_{a}\right)$ is evident, and also the necessity part is straightforward for the case of Lorentz spaces. Therefore, it suffices to demonstrate its sufficiency. For the space $L^{(p,r)}$, we can utilize Theorem 1.4 from \cite{Kam-Mal-Isr}. Now, let us consider the space $L^{p,r}$ and observe that $f^{*}(\infty)=0$ for all $f \in L^{p,r}$. We take $f$ and $f_{n}$ from $L^{p,r}$ such that $0\leq f_{n} \leq \left\vert f \right\vert $ and $f_{n} \rightarrow 0$ a.e. We can then apply Theorem 2.1 from \cite{Kol2018-Posit}. Using property $12,$ p. 67 from \cite{KPS82}, we can conclude that $f^{*}_{n}$ converges pointwise to zero. Finally, we can complete the proof by applying the Lebesgue dominated convergence theorem.

\end{proof}
Applying Theorem \ref{PK}, Theorem \ref{oc}, Lemma \ref{dens} and Lemma \ref{sep} we conclude:

\begin{corollary} \label{OC for HL}
Let $0<p<\infty $, $a \in \mathbb{R}$ and $q,r>0$. The Lorentz-Herz spaces $HL_{a,q}^{p,r}$ and $HL_{a,q}^{(p,r)}$ are order
continuous if and only if they are separable if and only if $r< \infty$ and $q< \infty$. Moreover, in that case, the set of simple functions having the finite measure is dense in these spaces.
\end{corollary}

By applying Lemma \ref{inlusions} and utilizing the inclusions between different Lorentz spaces $\left( L^{p,r}, \|\cdot\|_{L^{p,r}} \right)$ (see \cite{Krist}) as well as those between Lebesgue sequence spaces $l_{q}$, we arrive at the following result.

\begin{theorem} \label{wloz}
The following embedding results are valid:

\begin{enumerate}
\item[(A)] If $a \in \mathbb{R}$, $0<p< \infty$, $0 < r_{1} \leq r_{2} \leq
\infty$ and $0< q \leq \infty$, then 
\begin{equation*}
HL^{p,r_1}_{a,q} \hookrightarrow HL^{p,r_2}_{a,q}.
\end{equation*}
\item[(B)] If $a_2 \leq a_1$ and $0 < p,q,r \leq \infty$, then 
\begin{equation*}
HL^{p,r}_{a_1,q} \hookrightarrow HL^{p,r}_{a_2,q}.
\end{equation*}
\item[(C)] Suppose $\mu \left( \Omega\right) <\infty $. If $a \in \mathbb{R}$, $0 < r_1,r_2 \leq \infty$, $0 < q \leq
\infty$ and $0 < p_1 <p_2 < \infty$, then 
\begin{equation*}
HL^{p_2,r_2}_{a,q} \hookrightarrow HL^{p_1,r_1}_{a,q}.
\end{equation*}
\item[(D)] If $a \in \mathbb{R}$, $0 < r_1,r_2 \leq \infty$, $0 < q \leq
\infty$ and $0 < p_1 <p_2 < \infty$, then 
\begin{equation*}
HL^{p_2,r_2}_{a+Np_{0},q} \hookrightarrow HL^{p_1,r_1}_{a,q},
\end{equation*}
where $p_0=\frac{1}{p_1}-\frac{1}{p_2}>0$.
\item[(E)] If $a \in \mathbb{R}$, $0 < p,r \leq \infty$ and $0 < q_2 \leq
q_1 \leq \infty$, then 
\begin{equation*}
HL^{p,r}_{a,q_2} \hookrightarrow HL^{p,r}_{a,q_1}.
\end{equation*}
\end{enumerate}
\end{theorem}

\begin{proof}
    Notice only that conditions (C) and (D) are similar and form a certain dichotomy, namely, if we skip the assumption $\mu \left( \Omega\right) <\infty $, then we should change one of the parameter of Lebesgue sequence space. Hence we only need to prove condition (D). We have
\begin{eqnarray*}
\|f\chi_{\Omega_k}\|_{L^{p_1,r_1}} & = & \left[\int_{0}^{\infty}t^{\frac{r_1}{p_1}-1}\left(f\chi_{\Omega_k}\right)^{*r_1}dt\right]^{\frac{1}{r_1}} 
 =  \left[\int_{0}^{\mu(\Omega_k)}t^{\frac{r_1}{p_1}-1}\left(f\chi_{\Omega_k}\right)^{*r_1}dt\right]^{\frac{1}{r_1}} 
\end{eqnarray*}
Now proceeding as in \cite[p.54]{Krist}, we get
$$\|f\chi_{\Omega_k}\|_{L^{p_1,r_1}} \leq \frac{\mu(\Omega_k)^{\frac{1}{p_1}-\frac{1}{p_2}}}{\left(\frac{r_1}{p_1}-\frac{r_1}{p_2}\right)^{1/r_1}}\|f\chi_{\Omega_k}\|_{L^{p_2,\infty}}.$$
Since $\mu$ is Lebesgue measure, $\mu(\Omega_k) \leq C2^{kN},~ k \in I,$ for some constant $C$ depending on $N$ (see proof of Lemma \ref{bound}). Thus, 
$$\|f\chi_{\Omega_k}\|_{L^{p_1,r_1}} \leq C' 2^{kNp_0}\|f\chi_{\Omega_k}\|_{L^{p_2,\infty}},$$ where $p_0=\frac{1}{p_1}-\frac{1}{p_2}>0$ and $C'$ is a constant.
In consequence,
\begin{eqnarray*}
\|f\|^q_{HL^{p_1,r_1}_{a,q}}  =  \sum\limits_{k \in I} 2^{kaq}\|f\chi_{\Omega_k}\|^q_{L^{p_1,r_1}}  \leq  (C')^q \sum\limits_{k \in I} 2^{kaq+kNp_0q}\|f\chi_{\Omega_k}\|^q_{L^{p_2,\infty}}  = (C')^q\|f\|^q_{HL^{p_2,\infty}_{a+p_0N,q}}.
\end{eqnarray*}
Thus we have $HL^{p_2,\infty}_{a+Np_0,q} \hookrightarrow HL^{p_1,r_1}_{a,q}.$ From (A), $HL^{p_2,r_2}_{a+Np_0,q}\hookrightarrow HL^{p_2,\infty}_{a+Np_0,q}$. Therefore, we get the embedding 
$HL^{p_2,r_2}_{a+Np_0,q}\hookrightarrow HL^{p_1,r_1}_{a,q}$ without finite measure assumption. The case $q=\infty$ can be checked similarly.
\end{proof}

Below we denote by $p^{\prime },q^{\prime },r^{\prime }$ the conjugate exponents to $p,q,r$, respectively.
Applying Corollary \ref{Kothe-dual}, Remark \ref{oc-for-lor-l_q}, Corollary \ref{OC for HL} and Theorem 4.7, p.220, from \cite{Sharpley1988} we get

\begin{corollary}
\label{Kothe-dual1} [\textbf{The K\"{o}the dual of Lorentz-Herz spaces}] If $1 \leq r< \infty$, $1<p< \infty$ and $1<q< \infty$, then for $a \in \mathbb{R}$,
\begin{equation*}
\left( HL_{a,q}^{(p,r)}\right)^{*}=\left( HL_{a,q}^{(p,r)}\right)^{^{\prime
}} = HL_{-a,q^{\prime }}^{(p^{\prime },r^{\prime })}
\end{equation*}%
with equivalent norms. Moreover, if $1 \leq r \leq p < \infty$ and $1<q<
\infty$, then (upto equivalence of norms) 
\begin{equation*}
\left( HL_{a,q}^{p,r}\right)^{*} = \left( HL_{a,q}^{p,r}\right)^{^{\prime }}
= HL_{-a,q^{\prime }}^{p^{\prime },r^{\prime }} .
\end{equation*}%
Furthermore, in this case the H\"{o}lder's inequality (see Proposition \ref%
{HolderIq}) holds.
\end{corollary}

    Taking $p=r$ in the above corollary we get the characterization of dual of Herz spaces proved in  \cite{dual}, Theorem 2.1. Moreover, for suitable values of parameters as above, we conclude that the Lorentz-Herz spaces are reflexive. Applying Theorem \ref{BFS1}, we get

\begin{corollary} \begin{itemize}
\item[(i)] Suppose $1 \leq r< \infty$, $1<p< \infty$ and $1<q< \infty$. Then the
Lorentz-Herz space $HL_{a,q}^{(p,r)}$ is a Banach function space (in the
sense of Definition \ref{function norm}) if and only if for every $A\in
\Sigma $ with $\mu \left( A\right) <\infty $, the following conditions hold:%
\begin{enumerate}
\item[(a)] $\sum\limits_{k \in I } 2^{kaq}\mu(A \cap \Omega_k)^{\frac{q}{p}%
}< \infty$,
\item[(b)] $\sum\limits_{k \in I} 2^{-kaq^{^{\prime }}}\mu(A \cap \Omega_k)^{%
\frac{q^{^{\prime }}}{p^{^{\prime }}}} <\infty$.
\end{enumerate}
\item[(ii)] Suppose $1 \leq r \leq p < \infty$ and $1<q< \infty$. Then the
Lorentz-Herz space $HL_{a,q}^{p,r}$ is a Banach function space (in the sense
of Definition \ref{function norm}) if and only if for every $A\in \Sigma $
with $\mu \left( A\right) <\infty $, the following conditions hold: \newline
\begin{enumerate}
\item[(a)] $\sum\limits_{k \in I } 2^{kaq}\mu(A \cap \Omega_k)^{\frac{q}{p}%
}< \infty$,
\item[(b)] $\sum\limits_{k \in I} 2^{-kaq^{^{\prime }}}\mu(A \cap \Omega_k)^{%
\frac{q^{^{\prime }}}{p^{^{\prime }}}} <\infty$.
\end{enumerate}
\end{itemize}
\end{corollary}

\subsubsection{Real and complex interpolation}

\counterwithin{lemma}{subsubsection}


Let us recall basic facts about the real interpolation method (see {\cite%
{Bergh, Northholland})}. 
Let $Z=(Z_{0},Z_{1})$ be a pair of quasi-Banach spaces ({\cite[section 3.10,
3.11]{Bergh}}) with each of them continuously embedded in a Hausdroff
topological vector space $H$, we call $(Z_{0},Z_{1})$ a compatible couple.
For $z\in Z_{0}+Z_{1}$, the vector sum of $Z_{0}$ and $Z_{1}$, and for $%
0<t<\infty $, the K-functional is defined as 
\begin{equation*}
K(t,z;Z_{0},Z_{1})=\underset{z=z_{0}+z_{1}}{{\inf}}\{\Vert z_{0}\Vert
_{Z_{0}}+t\Vert z_{1}\Vert _{Z_{1}}:z_{i}\in Z_{i},i=0,1\}.
\end{equation*}%
Let $0<\theta <1$, $0<q\leq \infty $ or $0\leq \theta \leq 1$, $q=\infty $,
the space $(Z_{0},Z_{1})_{\theta ,q}=Z_{\theta ,q}$ consists of all $z\in
Z_{0}+Z_{1}$ for which the functional $\Vert z\Vert _{Z_{\theta ,q}}$ is
finite, where 
\begin{equation*}
\Vert z\Vert _{Z_{\theta ,q}}=%
\begin{cases}
\left( \int_{0}^{\infty }\left( \frac{K(t,z;Z_{0},Z_{1})}{t^{\theta }}%
\right) ^{q}\frac{dt}{t}\right) ^{\frac{1}{q}} & \mbox{if}~0<\theta
<1,0<q<\infty \\ 
\underset{0<t<\infty }{\sup }\frac{K(t,z;Z_{0},Z_{1})}{t^{\theta }} & %
\mbox{if}~0\leq \theta \leq 1,q=\infty .%
\end{cases}%
\end{equation*}%
$Z_{\theta ,q}$ is a quasi-Banach space with respect to the quasi-norm $%
\Vert \cdot \Vert _{Z_{\theta ,q}}$. For more details on these topics,
reader may consult \cite{Sharpley1988, Northholland}. Let $\mathbf{L}(X,Y)$
be the set of all bounded linear operators from a quasi-Banach space $X$ to
a quasi-Banach space $Y$.

\begin{definition}
\addtocounter{lemmaaux}{1} $($\cite{Northholland},~p.22$)$ An operator $R\in 
\mathbf{L}(X,Y)$ is said to be a retraction if there exists $S\in \mathbf{L}%
(Y,X)$ such that $RS=I$, the identity operator on $Y$. In this case, $S$ is
called as coretraction (belonging to $R$) and $Y$ is called retract of $X$.
\end{definition}

The following result stated in \cite{Northholland} has been proven for Banach spaces. Nonetheless, it can be extended (with the same proof) to quasi-Banach spaces by using the closed graph theorem, which is still applicable to quasi-Banach spaces (see  \cite[Theorem 1.6, p.10]{KPR84}). 


\begin{theorem}
$($\cite{Northholland}, p. 22$)$ \label{Theorem:retraction} Let $X=(X_{0},X_{1})$
and $Y=(Y_{0},Y_{1})$ be two compatible couples of Banach spaces. Assume
that $R\in \mathbf{L}(X_{0}+X_{1},Y_{0}+Y_{1})$ and $S\in \mathbf{L}%
(Y_{0}+Y_{1},X_{0}+X_{1})$ such that the restriction $S_{i}=S\big|_{Y_{i}}$,
is a coretraction of $\mathbf{L}(Y_{i},X_{i})$, and the restriction $R_{i}=R%
\big|_{X_{i}}$ is a retraction belonging to $\mathbf{L}(X_{i},Y_{i})$, $%
i=0,1 $. ($S_{i}$ belongs to $R_{i}$ in the sense of above definition.) Then 
$S$ represents the isomorphism from $Y_{\theta ,r}$ onto a complemented
subspace of $X_{\theta ,r}$  and
\begin{equation*}
\Vert y\Vert _{Y_{\theta ,r}}\approx \Vert Sy\Vert _{X_{\theta ,r}},
\end{equation*}%
where '$\approx $' represents the equivalence.
\end{theorem}

Now we apply real interpolation method to characterize the interpolation
spaces of Lorentz-Herz spaces. In particular, we get interpolation spaces
for non-homogeneous Herz spaces. For real interpolation of the homogeneous Herz
spaces, we refer readers to \cite{Young}. 

Suppose $B$ is a Banach space, $%
a\in \mathbb{R}$ and $0<q\leq \infty $, define 
\begin{equation*}
\ell _{q}^{a}({B})\coloneqq\left\{ \alpha :\alpha =(\alpha _{k})_{k\geq
-1},\alpha _{k}\in B~\text{and}~\left[ \sum\limits_{k\geq -1}2^{kaq}\Vert
\alpha _{k}\Vert _{B}^{q}\right] ^{\frac{1}{q}}<\infty \right\} ,
\end{equation*}%
with usual modification for $q=\infty $. Then it is easy to see that $\ell
_{q}^{a}(B)$ is a quasi-Banach space (a particular case of K\"{o}the-Bochner sequence space). We recall the following interpolation
result for $\ell _{q}^{a}(B)$ (see \cite{Bergh}, Theorem 5.6.1 and Theorem 5.6.2 in
Section 5.6 and \cite{Northholland}, Section 1.18.1).

\begin{theorem}
\label{Theorem:intp-main} Let $0< \theta <1$, $0<q_0,q_1,q \leq \infty$ and $%
a_0,a_1,a \in \mathbb{R}$.

\begin{enumerate}
\item[(1)] If $a_0 \neq a_1$ and $a=(1-\theta)a_0+\theta a_1$, then 
\begin{equation*}
\left(\ell_{q_0}^{a_0}(B),\ell_{q_1}^{a_1}(B)\right)_{\theta,q}=%
\ell_{q}^{a}(B).
\end{equation*}

\item[(2)] If $\frac{1}{q}=\frac{1-\theta}{q_0}+\frac{\theta}{q_1}$, then 
\begin{equation*}
\left(\ell_{q_0}^{a}(B),\ell_{q_1}^{a}(B)\right)_{\theta,q}=\ell_{q}^{a}(B).
\end{equation*}

\item[(3)] If $0<q_0,q_1< \infty$ and $(B_0,B_1)$ is a compatible couple of
Banach spaces, then 
\begin{equation*}
\left(\ell_{q_0}^{a_0}(B_0),\ell_{q_1}^{a_1}(B_1)\right)_{\theta,q}=%
\ell_{q}^{a}\left((B_0,B_1)_{\theta,q}\right)
\end{equation*}
provided that $a=(1-\theta)a_0+\theta a_1$ and $\frac{1}{q}=\frac{1-\theta}{%
q_0}+\frac{\theta}{q_1}$.
\end{enumerate}
\end{theorem}

We also recall the interpolation result for Lorentz spaces (\cite{Bergh},
Section 5.3 and \cite{Northholland}, Section 1.18.6).

\begin{theorem}
\label{Theorem:Lorentz-intp} Assume that $0<\theta <1$, $0 < p_{0}\neq p_{1}
\leq \infty $, $0 < r_{0},r_{1},r\leq \infty $. Then 
\begin{equation*}
\left( L^{p_{0},r_{0}},L^{p_{1},r_{1}}\right) _{\theta ,r}=L^{p,r},
\end{equation*}%
where $\frac{1}{p}=\frac{1-\theta }{p_{0}}+\frac{\theta }{p_{1}}$. This
formula is also true in the case $p_0=p_1=p$, provided that $\frac{1}{r}=%
\frac{1-\theta}{r_0}+\frac{\theta}{r_1}$.
\end{theorem}

\begin{definition}
Let mappings $S:HL_{a,q}^{p,r}\rightarrow \ell _{q}^{a}\left( L^{p,r}\right) 
$ and $R:\ell _{q}^{a}\left( L^{p,r}\right) \rightarrow HL_{a,q}^{p,r}$ be
defined by%
\begin{equation*}
S\left( f\right) =\left( f\chi _{\Omega _{k}}\right) _{k}\text{ for each }%
f\in HL_{a,q}^{p,r}
\end{equation*}%
and%
\begin{equation*}
R\left( f\right) =\sum_{k}f_{k}\chi _{\Omega _{k}}\text{ for each }f\in \ell
_{q}^{a}\left( L^{p,r}\right) ,f=\left( f_{k}\right) _{k}.
\end{equation*}
\end{definition}

\begin{lemma}
\label{lem:isometry}We have $S\in \mathbf{L}(HL_{a,q}^{p,r},\ell
_{q}^{a}\left( L^{p,r}\right) )$ and $R\in \mathbf{L}(\ell _{q}^{a}\left(
L^{p,r}\right) ,HL_{a,q}^{p,r})$. Moreover, $S$ is an isometry and $R\circ
S=I$, the identity map.
\end{lemma}

\begin{proof}
The proof is direct.
\end{proof}

Now we are in a position to formulate a real interpolation results for
Lorentz-Herz spaces.

\begin{theorem}
\label{myinterpolation}

\begin{itemize}
\item[(1)] Assume that $0<\theta <1$, $1\leq r\leq p<\infty $, $%
0<q_{0},q_{1},q\leq \infty $ and $a_{0},a_{1},a\in \mathbb{R}$. Then the
following interpolation formulae are valid:\newline

\begin{itemize}
\item[(i)] If $a_{0}\neq a_{1}$ and $a=(1-\theta )a_{0}+\theta a_{1}$, then 
\begin{equation*}
\left( HL_{a_{0},q_{0}}^{p,r},HL_{a_{1},q_{1}}^{p,r}\right) _{\theta
,q}=HL_{a,q}^{p,r}.
\end{equation*}

\item[(ii)] If $\frac{1}{q}=\frac{1-\theta }{q_{0}}+\frac{\theta }{q_{1}}$,
then 
\begin{equation*}
\left( HL_{a,q_{0}}^{p,r},HL_{a,q_{1}}^{p,r}\right) _{\theta
,q}=HL_{a,q}^{p,r}.
\end{equation*}
\end{itemize}

\item[(2)] If $0<q_{0},q_{1},q<\infty $, $1\leq p_{0}\neq p_{1}\leq \infty $%
, $1\leq r_{0},r_{1}\leq \infty $ and $1\leq r_{i}\leq p_{i}<\infty ,~i=0,1$%
, then 
\begin{equation*}
\left( HL_{a_{0},q_{0}}^{p_{0},r_{0}},HL_{a_{1},q_{1}}^{p_{1},r_{1}}\right)
_{\theta ,q}=HL_{a,q}^{p,q}
\end{equation*}%
provided that $a=(1-\theta )a_{0}+\theta a_{1}$, $\frac{1}{p}=\frac{1-\theta 
}{p_{0}}+\frac{\theta }{p_{1}}$ and $\frac{1}{q}=\frac{1-\theta }{q_{0}}+%
\frac{\theta }{q_{1}}$. In addition, if $p=q$, then 
\begin{equation*}
\left( HL_{a_{0},q_{0}}^{p_{0},r_{0}},HL_{a_{1},q_{1}}^{p_{1},r_{1}}\right)
_{\theta ,q}=K_{a,p}^{p}.
\end{equation*}

\item[(3)] If $0<q_{0},q_{1}<\infty ,$ $a=(1-\theta )a_{0}+\theta a_{1}$ and 
$\frac{1}{q}=\frac{1-\theta }{q_{0}}+\frac{\theta }{q_{1}}$, then 
\begin{equation*}
\left( K_{a_{0},q_{0}}^{p_{0}},K_{a_{1},q_{1}}^{p_{1}}\right) _{\theta
,q}=HL_{a,q}^{p,q}
\end{equation*}%
provided that $1\leq p_{0}\neq p_{1}\leq \infty $ and $\frac{1}{p}=\frac{%
1-\theta }{p_{0}}+\frac{\theta }{p_{1}}$.
\end{itemize}
\end{theorem}

\begin{proof}
(1) In view of Theorem \ref{Theorem:retraction},
Theorem \ref{Theorem:intp-main} and Lemma \ref{lem:isometry}, we have 
\begin{equation*}
\Vert f\Vert _{\left( HL_{a_{0},q_{0}}^{p,r},HL_{a_{1},q_{1}}^{p,r}\right)
_{\theta ,q}}\approx \Vert S(f)\Vert _{\left( \ell
_{q_{0}}^{a_{0}}(L^{p,r}),\ell _{q_{1}}^{a_{1}}(L^{p,r})\right) _{\theta
,q}}=\Vert \left( f\chi _{\Omega _{k}}\right) _{k}\Vert _{\ell
_{q}^{a}(L^{p,r})}=\Vert f\Vert _{HL_{a,q}^{p,r}},
\end{equation*}%
i.e., 
\begin{equation*}
\left( HL_{a_{0},q_{0}}^{p,r},HL_{a_{1},q_{1}}^{p,r}\right) _{\theta
,q}=HL_{a,q}^{p,r}.
\end{equation*}%
This proves (i), the proof of (ii) is similar. The second part $(2)$ can be proved by using the same technique along with Theorem \ref{Theorem:Lorentz-intp}. Further, we have
\begin{equation*}
\left( K_{a_{0},q_{0}}^{p_{0}},K_{a_{1},q_{1}}^{p_{1}}\right) _{\theta
,q}=\left(
HL_{a_{0},q_{0}}^{p_{0},p_{0}},HL_{a_{1},q_{1}}^{p_{1},p_{1}}\right)
_{\theta ,q}=HL_{a,q}^{p,q},
\end{equation*}
which proves (3). This completes the proof.
\end{proof}Theorem \ref{myinterpolation} shows that, with appropriately
chosen parameters, the interpolation spaces of Lorentz-Herz spaces are again
Lorentz-Herz spaces. Thus, they are (in some sense) closed under
interpolation. 
In addition, the Lorentz-Herz space $HL_{a,q}^{p,q}$ appears as an
interpolation space between the non-homogeneous Herz spaces. Furthermore, if 
$p=r$ and $p_{i}=r_{i}$, $i=0,1$, we get real interpolation spaces of the
non-homogeneous Herz spaces.

From now until the end of this subsection we consider complex Banach spaces. We briefly recall the complex method of interpolation for Banach spaces (see
for example \cite{Bergh}). Given a couple $(A_{0},A_{1})$ of complex Banach spaces, consider
the linear space $\mathcal{F}(A_{0},A_{1})$ of all functions $f:\mathbb{C}%
\rightarrow A_{0}+A_{1}$, that are bounded and continuous on the closed
strip $\bar{S}=\{z:0\leq Re(z)\leq 1\}$, analytic on the open strip $%
S=\{z:0<Re(z)<1\}$ and the functions $t\rightarrow f(s+it)$ are continuous
from $\mathbb{R}$ into $A_{s}$ with $\lim_{\left\vert t\right\vert
\rightarrow \infty }f(s+it)=0$ $(s=0,1).$ The space $\mathcal{F}%
(A_{0},A_{1})$, is a Banach space with respect to the norm 
\begin{equation*}
\Vert f\Vert _{\mathcal{F}(A_{0},A_{1})}=\max \left( \underset{t\in \mathbb{R%
}}{\sup }\Vert f(it)\Vert _{A_{0}},\underset{t\in \mathbb{R}}{\sup }\Vert
f(1+it)\Vert _{A_{1}}\right) .
\end{equation*}%
For $0<\theta <1$, the space $[A_{0},A_{1}]_{\theta }$ consisting of all $%
x\in A_{0}+A_{1}$ such that $x=f(\theta )$ for some $f\in \mathcal{F}%
(A_{0},A_{1}),$ equipped with the norm 
\begin{equation*}
\Vert x\Vert _{\lbrack \theta ]}=\inf \{\Vert f\Vert _{\mathcal{F}%
(A_{0},A_{1})}:f(\theta )=x,~f\in \mathcal{F}(A_{0},A_{1})\},
\end{equation*}%
is a Banach space. This space is intermediate for the couple $(A_{0},A_{1})$
and the construction $C_{\theta }:(A_{0},A_{1})\rightarrow \lbrack
A_{0},A_{1}]_{\theta }$ is the interpolation functor for bounded, linear
operators. The space $[A_{0},A_{1}]_{\theta }$ is called the complex
interpolation space.

We also recall the following result on complex interpolation of Lebesgue-Bochner sequence spaces (\cite{Bergh}, Theorem 5.6.3, p. 123). Below we exclude the case $ q_0=q_1 = \infty$, because then the complex interpolation does not work as we would like to (see  \cite[p. 123, (16)]{Northholland}).

\begin{theorem}
\label{Complex} Let $0< \theta <1$, $1 \leq q_0,q_1 \leq \infty$ (except the case $ q_0=q_1 = \infty$) and $%
a_0,a_1 \in \mathbb{R}$. Then 
\begin{equation*}
\left[\ell_{q_0}^{a_0}(A_0),\ell_{q_1}^{a_1}(A_1)\right]_{\theta}=%
\ell_{q}^{a}\left([A_0,A_1]_{\theta} \right),
\end{equation*}
where $\frac{1}{q}=\frac{1-\theta}{q_0}+\frac{\theta}{q_1}$ and $%
a=(1-\theta)a_0+\theta a_1$.
\end{theorem}

Now we give our complex interpolation result for Lorentz-Herz spaces.

\begin{theorem}
Let $0<\theta<1$, $1<r_i \leq p_i<\infty~(i=0,1)$, $1 \leq q_0,q_1 \leq
\infty$ (except the case $ q_0=q_1 = \infty$) and $a_0,a_1 \in \mathbb{R}$. Then 
\begin{equation*}
\left[ HL^{p_0,r_0}_{a_0,q_0},HL^{p_1,r_1}_{a_1,q_1} \right]%
_{\theta}=HL^{p,r}_{a,q},
\end{equation*}
where $\frac{1}{p}=\frac{1-\theta}{p_0}+\frac{\theta}{p_1}$, $\frac{1}{q}=%
\frac{1-\theta}{q_0}+\frac{\theta}{q_1}$, $\frac{1}{r}=\frac{1-\theta}{r_0}+%
\frac{\theta}{r_1}$ and $a=(1-\theta)a_0+\theta a_1$.
\end{theorem}

\begin{proof}
We only need to observe that $\left[ L^{p_{0},r_{0}},L^{p_{1},r_{1}}\right]
_{\theta }=L^{p,r}$ (see \cite[Section 2]{FSOLST}). Then the result is
immediate from Theorem \ref{Complex}, Theorem \ref{Theorem:retraction} (which still true for any interpolation functor, see \cite{Northholland}, p. 22) ) and Lemma %
\ref{lem:isometry}.
\end{proof}

Taking above $p_{0}=r_{0}$ and $p_{1}=r_{1}$ we get immediately the result on complex interpolation on Herz spaces (see \cite{dual, Coifman}).
\begin{corollary}
 Let $0<\theta<1$, $1< p_0<\infty$, $1< p_1<\infty$, $1 \leq q_0,q_1 \leq
\infty$ (except the case $ q_0=q_1 = \infty$) and $a_0,a_1 \in \mathbb{R}$. Then 
\begin{equation*}
\left[ K^{p_0}_{a_0,q_0},K^{p_1}_{a_1,q_1} \right]%
_{\theta}=K^{p}_{a,q},
\end{equation*}
where $\frac{1}{p}=\frac{1-\theta}{p_0}+\frac{\theta}{p_1}$, $\frac{1}{q}=%
\frac{1-\theta}{q_0}+\frac{\theta}{q_1}$  and $a=(1-\theta)a_0+\theta a_1$.
   
\end{corollary}

\subsubsection{Boundedness of some operators}

In this section, we prove boundedness of a wide class of sublinear operators
on Lorentz-Herz spaces. 
The class of operators whose kernels satisfy the standard size condition
\begin{equation}
   |k(x,y)|\leq \frac{C}{|x-y|^{N}},
\end{equation} have been the subject of extensive research due to their wide range of applications in harmonic analysis (e.g., see \cite{Fefferman1979PAMS, dual,ossc1,sizecondition,Stein1957PAMS,Stein-1970}). This condition ensures that the operators are well-defined and possess desirable properties such as continuity, boundedness, and compactness. Singular integral operators of such kind are essential in the study of the fine structure of functions and their properties such as regularity, smoothness, and decay at infinity.
In particular, the authors in \cite{dual, sublinear1} proved that if a
sublinear operator $T$ is bounded on $L^{p}(\mathbb{R}^{N})$ and is
satisfying the size condition 
\begin{equation}
|Tf(x)|\leq C\int_{\mathbb{R}^{N}}\frac{|f(y)|}{|x-y|^{N}}d\mu (y),~x\notin
supp(f)  \label{sizeconditin}
\end{equation}%
for all $f\in L^{1}(\mathbb{R}^N)$ with compact support, where $C$ is the
constant independent of $f$ and $x$, then $T$ is bounded on $K_{a,q}^{p}(%
\mathbb{R}^{N})$. We extend this result to the spaces $HL_{a,q}^{p,r}(\mathbb{R}^N)$, which are more general than the corresponding Herz spaces discussed in \cite{dual}. Our result is even slightly stronger since we do not assume $T$ to be linear for $a=0$. Moreover, for $r=\infty$, we obtain the corresponding result for weak Herz spaces, which is not available in the existing literature.


From this point onwards we follow the notation $a\lesssim b$ for $a\leq Cb$, where $C$ is a constant, independent of appropriate quantities. Before we proceed to the main theorem of this section, we require the following lemma.

%
%
%
%

\begin{lemma}
\label{bound} Let $\Omega=\mathbb{R}^N$ and $k,v \geq -1$ be integers. For $%
1<p<\infty$ and $1\leq r \leq \infty$, 
\begin{equation*}
2^{-kN}\|\chi_{\Omega_k}\|_{{p,r}}\|\chi_{\Omega_v}\|_{{p^{\prime
},r^{\prime }}} \lesssim 2^{\frac{N}{p^{\prime }}(v-k)}.
\end{equation*}
\end{lemma}

\begin{proof}
It is quite easy to see that, for $0<p<\infty $, $0<r<\infty $ and a
measurable set $A\subseteq \mathbb{R}^{N}$, $\Vert \chi _{A}\Vert
_{p,r}=\left( \frac{p}{r}\right) ^{\frac{1}{r}}\mu (A)^{\frac{1}{p}}$ and $%
\Vert \chi _{A}\Vert _{p,\infty }=\mu (A)^{\frac{1}{p}}$. Moreover, in
closed form, the volume of ball of radius $R$ in $\mathbb{R}^{N}$ is given by 
$\frac{\pi ^{N/2}}{\Gamma (\frac{n}{2}+1)}R^{N}$, where $\Gamma $ is the
Euler's Gamma function. Thus 
\begin{equation*}
\mu (\Omega _{k})=%
\begin{cases}
\frac{\pi^{N/2}}{\Gamma(\frac{N}{2}+1)}\left(1-\frac{1}{2^N}\right)2^{kN},~ 
\text{if} ~k \geq 0 \\ 
\frac{\pi^{N/2}}{\Gamma(\frac{N}{2}+1)}2^{-N}, ~\text{if}~ k = -1.%
\end{cases}%
\end{equation*}%
Therefore, it follows that there is a constant $\gamma $, depending upon $%
p,r $ and $N$, such that 
\begin{equation*}
2^{-kN}\Vert \chi _{\Omega _{k}}\Vert _{{p,r}}\Vert \chi _{\Omega _{v}}\Vert
_{{p^{\prime },r^{\prime }}}\leq \gamma \cdot 2^{\frac{N}{p^{\prime }}(v-k)}.
\end{equation*}%
Equivalently, 
\begin{equation*}
2^{-kN}\Vert \chi _{\Omega _{k}}\Vert _{{p,r}}\Vert \chi _{\Omega _{v}}\Vert
_{{p^{\prime },r^{\prime }}}\lesssim 2^{\frac{N}{p^{\prime }}(v-k)}.
\end{equation*}%
\end{proof}

\begin{theorem}
\label{BSO} Let $T$ be a sublinear operator satisfying the size condition (%
\ref{sizeconditin}) for all $f\in L^{1}(\mathbb{R}^N)$ with compact
support. Suppose that $1<p<\infty $, $1\leq q,r\leq \infty $ and $T$ is
bounded on $L^{p,r}(\mathbb{R}^N)$. Then for $\frac{-N}{p}<a<\frac{N}{%
p^{\prime }}$, $T$ is bounded on $HL_{a,q}^{p,r}(\mathbb{R}^N)$. 
\end{theorem}

\begin{proof}
Let $f\in HL_{a,q}^{p,r}(\mathbb{R}^N)$. Since $T$ is sublinear, we have 
\begin{eqnarray*}
\Vert Tf\Vert _{HL_{a,q}^{p,r}} &=&\left[ \sum\limits_{k\in I}2^{kaq}\Vert
\chi _{\Omega _{k}}\cdot Tf\Vert _{{p,r}}^{q}\right] ^{\frac{1}{q}} \leq \left[ \sum\limits_{k\in I}2^{kaq}\Vert \chi _{\Omega _{k}}\cdot
\sum\limits_{v\in I}\big\vert T(f\chi _{\Omega _{v}})\big\vert\Vert _{{(p,r)}%
}^{q}\right] ^{\frac{1}{q}} \\
&\leq &\left[ \sum\limits_{k\in I}2^{kaq}\left( \sum\limits_{v\in I}\Vert
\chi _{\Omega _{k}}\big\vert T(f\chi _{\Omega _{v}})\big\vert\Vert _{{(p,r)}%
}\right) ^{q}\right] ^{\frac{1}{q}} \\
&\lesssim &\left[ \sum\limits_{k\in I}2^{kaq}\left( \sum\limits_{v\in
I}\Vert \chi _{\Omega _{k}}T(f\chi _{\Omega _{v}})\Vert _{{p,r}}\right) ^{q}%
\right] ^{\frac{1}{q}}.
\end{eqnarray*}%
The last sum can be expressed as 
\[
\displaystyle\left[ \sum\limits_{k\in I}2^{kaq}\left(
\sum\limits_{v=-1}^{k-2}\Vert \chi _{\Omega _{k}}T(f\chi _{\Omega
_{v}})\Vert _{{p,r}}+\sum\limits_{v=k-1}^{k+1}\Vert \chi _{\Omega
_{k}}T(f\chi _{\Omega _{v}})\Vert _{{p,r}}+\sum\limits_{v\geq k+2}\Vert \chi
_{\Omega _{k}}T(f\chi _{\Omega _{v}})\Vert _{{p,r}}\right) ^{q}\right] ^{%
\frac{1}{q}}.
\]%
Applying Minkowski's inequality, we get 
\begin{eqnarray}\label{threesums}
\Vert Tf\Vert _{HL_{a,q}^{p,r}} & \lesssim & \left[ \sum\limits_{k \in I } 2^{kaq} \left(\sum\limits_{v = -1}^{k-2}\|\chi_{\Omega_k}T(f\chi_{\Omega_v})\|_{{p,r}}
\right)^{q} \right]^{\frac{1}{q}} + \left[ \sum\limits_{k \in I }
2^{kaq} \left( \sum\limits_{v=k-1}^{k+1}\|\chi_{\Omega_k}T(f\chi_{\Omega_v})\|_{{p,r}} \right)^{q} \right]^{\frac{1}{q}} \nonumber \\ & + & \left[ \sum\limits_{k \in I } 2^{kaq} \left(\sum\limits_{v \geq k+2}\|\chi_{\Omega_k}T(f\chi_{\Omega_v})\|_{{p,r}}\right)^{q} \right]^{\frac{1}{q}} \nonumber \\ & := & S_{1}+S_{2}+S_{3}.
\end{eqnarray}
%

Now observe that for $v\leq k-2$ and a.e. $x\in \Omega _{k}$, the size
condition (\ref{sizeconditin}) and the H\"{o}lder inequality for $L^{p,r}$
spaces (recall that $\left( L^{p,r}\right) ^{\prime }=L^{p^{\prime
},r^{\prime }}$ for $1<p<\infty $ and $1\leq r\leq \infty ,$ see Theorem 4 in 
\cite{KLM-2019}, for even more general case) imply that 
\[
|T(f\chi _{\Omega _{v}})(x)|\lesssim 2^{-kN}\Vert f\chi _{\Omega _{v}}\Vert
_{{p,r}}\Vert \chi _{\Omega _{v}}\Vert _{{p^{\prime },r^{\prime }}}.
\]%
Therefore, 
\[
S_{1}\lesssim \left[ \sum\limits_{k\in I}2^{kaq}\left(
\sum\limits_{v=-1}^{k-2}2^{-kN}\Vert \chi _{\Omega _{k}}\Vert _{{p,r}}\Vert
f\chi _{\Omega _{v}}\Vert _{{p,r}}\Vert \chi _{\Omega _{v}}\Vert _{{%
p^{\prime },r^{\prime }}}\right) ^{q}\right] ^{\frac{1}{q}}.
\]%
Using Lemma \ref{bound} , we get 
\begin{eqnarray*}
S_{1} &\lesssim &\left[ \sum\limits_{k\in I}2^{kaq}\left(
\sum\limits_{v=-1}^{k-2}2^{va}\Vert f\chi _{\Omega _{v}}\Vert _{{p,r}}\cdot
2^{\frac{N}{p^{\prime }}(v-k)-va}\right) ^{q}\right] ^{\frac{1}{q}} \\
&\lesssim &\left[ \sum\limits_{k\geq -1}\left(
\sum\limits_{v=-1}^{k-2}2^{va}\Vert f\chi _{\Omega _{v}}\Vert _{{p,r}}\cdot
2^{\beta (v-k)}\right) ^{q}\right] ^{\frac{1}{q}},
\end{eqnarray*}%
where $\beta =\frac{N}{p^{\prime }}-a>0$. Using H\"{o}lder's inequality for
the inner sum, we obtain 
\begin{eqnarray*}
S_{1} &\lesssim &\left[ \sum\limits_{k\geq -1}\left\{ \left(
\sum\limits_{v=-1}^{k-2}2^{vaq}\Vert f\chi _{\Omega _{v}}\Vert _{{p,r}%
}^{q}\cdot 2^{q\beta (v-k)/2}\right) \left(
\sum\limits_{v=-1}^{k-2}2^{q^{\prime }\beta (v-k)/2}\right) ^{\frac{q}{%
q^{\prime }}}\right\} \right] ^{\frac{1}{q}} \\
&\lesssim &\left[ \sum\limits_{k\geq -1}\left(
\sum\limits_{v=-1}^{k-2}2^{vaq}\Vert f\chi _{\Omega _{v}}\Vert _{{p,r}%
}^{q}\cdot 2^{q\beta (v-k)/2}\right) \right] ^{\frac{1}{q}}.
\end{eqnarray*}%
Interchanging the order of summations, we conclude that 
\[
S_{1}\lesssim \left[ \sum\limits_{v\geq -1}\left( \sum\limits_{k\geq
v+2}2^{vaq}\Vert f\chi _{\Omega _{v}}\Vert _{{p,r}}^{q}\cdot 2^{q\beta
(v-k)/2}\right) \right] ^{\frac{1}{q}}.
\]%
Therefore, 
\begin{equation}
S_{1}\lesssim \Vert f\Vert _{HL_{a,q}^{p,r}}.  \label{S1}
\end{equation}%
Now $\displaystyle S_{2}=\left[ \sum\limits_{k\in I}2^{kaq}\left(
\sum\limits_{v=k-1}^{k+1}\Vert \chi _{\Omega _{k}}T(f\chi _{\Omega
_{v}})\Vert _{{p,r}}\right) ^{q}\right] ^{\frac{1}{q}}\leq \left[
\sum\limits_{k\in I}2^{kaq}\left( \sum\limits_{v=k-1}^{k+1}\Vert T(f\chi
_{\Omega _{v}})\Vert _{{p,r}}\right) ^{q}\right] ^{\frac{1}{q}}.$ 
Using the boundedness of $T$ on $L^{p,r}$ and Minkowski's inequality, we get 
\begin{eqnarray*}
S_{2} &\lesssim &\left[ \sum\limits_{k\in I}2^{kaq}\left(
\sum\limits_{v=k-1}^{k+1}\Vert f\chi _{\Omega _{v}}\Vert _{{p,r}}\right)
^{q}\right] ^{\frac{1}{q}} \\
&\lesssim &\left[ \sum\limits_{k\in I}2^{kaq}\Vert f\chi _{\Omega
_{k-1}}\Vert _{{p,r}}^{q}\right] ^{\frac{1}{q}}+\left[ \sum\limits_{k\in
I}2^{kaq}\Vert f\chi _{\Omega _{k}}\Vert _{{p,r}}^{q}\right] ^{\frac{1}{q}}+%
\left[ \sum\limits_{k\in I}2^{kaq}\Vert f\chi _{\Omega _{k+1}}\Vert _{{p,r}%
}^{q}\right] ^{\frac{1}{q}} \\
&\lesssim &\left[ \sum\limits_{k\in I}2^{kaq}\Vert f\chi _{\Omega
_{k}}\Vert _{{p,r}}^{q}\right] ^{\frac{1}{q}}.
\end{eqnarray*}%
Thus, 
\begin{equation}
S_{2}\lesssim \Vert f\Vert _{HL_{a,q}^{p,r}}.  \label{S2}
\end{equation}%
Finally, 
for $v\geq k+2$ and a.e. $x$ in $\Omega _{k}$, the size condition (\ref%
{sizeconditin}) and H\"{o}lder's inequality for Lorentz spaces imply that 
\[
|T(f\chi _{\Omega _{v}})(x)|\lesssim 2^{-vN}\Vert f\chi _{\Omega _{v}}\Vert
_{{p,r}}\Vert \chi _{\Omega _{v}}\Vert _{{p^{\prime },r^{\prime }}}.
\]%
Therefore, by similar calculations as in $S_{1}$, we get 
\begin{eqnarray*}
S_{3} &\lesssim &\left[ \sum\limits_{k\in I}2^{kaq}\left( \sum\limits_{v\geq
k+2}2^{va}\Vert f\chi _{\Omega _{v}}\Vert _{{p,r}}\cdot 2^{\frac{N}{p}%
(k-v)-va}\right) ^{q}\right] ^{\frac{1}{q}} \\
&\lesssim &\left[ \sum\limits_{k\geq -1}\left( \sum\limits_{v\geq
k+2}2^{va}\Vert f\chi _{\Omega _{v}}\Vert _{{p,r}}\cdot 2^{\beta ^{\prime
}(k-v)}\right) ^{q}\right] ^{\frac{1}{q}},
\end{eqnarray*}%
where $\beta ^{\prime }=\frac{N}{p}+a>0$. Using H\"{o}lder's inequality for
the inner sum, we obtain
\begin{eqnarray*}
S_{3} &\lesssim &\left[ \sum\limits_{k\geq -1}\left\{ \left(
\sum\limits_{v\geq k+2}2^{vaq}\Vert f\chi _{\Omega _{v}}\Vert _{{p,r}%
}^{q}\cdot 2^{q\beta ^{\prime }(k-v)/2}\right) \left( \sum\limits_{v\geq
k+2}2^{q^{\prime }\beta (k-v)/2}\right) ^{\frac{q}{q^{\prime }}}\right\} %
\right] ^{\frac{1}{q}} \\
&\lesssim &\left[ \sum\limits_{k\geq -1}\left( \sum\limits_{v\geq
k+2}2^{vaq}\Vert f\chi _{\Omega _{v}}\Vert _{{p,r}}^{q}\cdot 2^{q\beta
^{\prime }(k-v)/2}\right) \right] ^{\frac{1}{q}}.
\end{eqnarray*}%
Interchanging the order of summations, we have 
\begin{equation}
S_{3}\lesssim \left[ \sum\limits_{v\geq -1}\left(
\sum\limits_{k=-1}^{v-2}2^{vaq}\Vert f\chi _{\Omega _{v}}\Vert _{{p,r}%
}^{q}\cdot 2^{q\beta ^{\prime }(k-v)/2}\right) \right] ^{\frac{1}{q}%
}\lesssim \Vert f\Vert _{HL_{a,q}^{p,r}}.  \label{S3}
\end{equation}%
Combining the estimates (\ref{threesums}), (\ref{S1}), (\ref{S2}) and (\ref%
{S3}), we obtain $\Vert Tf\Vert _{HL_{a,q}^{p,r}}\lesssim \Vert f\Vert
_{HL_{a,q}^{p,r}}$. This completes the proof.
\end{proof}

\begin{remark}
The operator $T$ is bounded on $HL^{p,r}_{a,q}$ if and only if $T$ is
bounded on $HL^{(p,r)}_{a,q}$ for $1<p< \infty$, $1 \leq r \leq \infty$.
\end{remark}

The sublinear operators satisfying condition (\ref{sizeconditin}) have been studied by many authors, we refer to book \cite{Stein-1970}, see also \cite{BCZO,dual,sublinear1}. This condition is satisfied by several operators of critical importance in harmonic analysis. Some examples are
Caledr\'on-Zygmund operators, Hardy-Littlewood maximal operator,
Bochner-Riesz means, Carleson's maximal operator, and certain singular
integrals $($see \cite{ossc1, ossc2,sizecondition}$)$.

Since the Hardy-Littlewood maximal operator is bounded on $L^{p,r}(\mathbb{R}%
^N)$ for $1<p<\infty,~1 \leq r \leq \infty$ (see \cite{BMO}, Corollary 4 )
and the Caledr\'on-Zygmund operators are bounded on $L^{p,r}(\mathbb{R}^N)$
for $1<p < \infty,~1 \leq r < \infty$ (see \cite[Theorem 1.2]{BCZO}), we
have the following conclusions from Theorem \ref{BSO}.

\begin{corollary}
Let $1<p< \infty$ and $\frac{-N}{p}<a< \frac{N}{p^{\prime }}$.

\begin{enumerate}
\item If $1 \leq q,r \leq \infty$, then the Hardy-Littlewood maximal
operator is bounded on $HL^{p,r}_{a,q}(\mathbb{R}^N)$.

\item If $1 \leq r < \infty$ and $1 \leq q \leq \infty$, then the
Caledr\'on-Zygmund operators are bounded on $HL^{p,r}_{a,q}(\mathbb{R}^N)$.
\end{enumerate}
\end{corollary}


%
%
%
%
%
%

\section*{Acknowledgements}

The first author (M. Ashraf Bhat) is thankful to Prime Minister's Research
Fellowship (PMRF) program for the fellowship (PMRF ID: 2901480).

The second author's research (Pawe{\l} Kolwicz) was supported by the Poznan University of Technology under  Grant no. 0213/SBAD/0118.

\appendix

\section{Appendix A: Orlicz-Herz spaces}
The two most important generalizations of Lebesgue spaces $L^{p}$ are
Lorentz spaces $L^{p,r}$ and Orlicz spaces $L^{\Phi }.$ We will now consider the construction $E(\mathcal{X})$ using the family $\mathcal{X}=(X_{\alpha})_{\alpha\in I}$ of quasi-normed spaces as the Orlicz spaces $L^{\Phi}$, and $E$ as the weighted Lebesgue sequence space $l_{q}(w_{a})$ over $I$ with $w_{a}=w_{a}(k)=2^{ak}$. We will explore the basic properties of the resulting space, which we naturally call the Orlicz-Herz spaces denoted by $HO_{a,q}^{\Phi}$.


\begin{definition}
\label{phi} A function $\Phi :[0,\infty )\rightarrow \lbrack 0,\infty )$ is
called an \textit{Orlicz function} if $\Phi $ is non-decreasing, vanishing
and right continuous at $0$, continuous on $(0,\infty )$ and satisfying the
condition $\lim_{u\rightarrow \infty }\Phi (u)=\infty $. Any Orlicz function 
$\Phi $ determines a functional $I_{\Phi }:L^{0} \rightarrow \lbrack
0,\infty ]$ defined by the formula $I_{\Phi }(f)=\int_{\Omega }\Phi
(|f(t)|)d\mu (t)$. The order ideal 
\begin{equation*}
L^{\Phi }(\mu )=\{f\in L^{0}:I_{\Phi }(rf)<\infty \text{ for some }%
r>0\}
\end{equation*}%
in $L^{0}$ is called an Orlicz space. We consider the space $L^{\Phi
}(\mu )$ with the following functional: 
\begin{equation} \label{MO-norm}
\Vert f\Vert _{\Phi }=\inf\{\lambda >0:I_{\Phi }(f/\lambda )\leq
1\}.  
\end{equation}
\end{definition}

Denote by $\alpha _{\Phi }$ the lower Matuszewska-Orlicz index $\alpha
_{\Phi }$ (see \cite{Fo-Hu-Ko}, \cite{Kam-Mal-Per-2003}, \cite{Kam-Zyl} for
the definition). Recall that $\alpha _{\Phi }>0$ if and only if the
functional (\ref{MO-norm}) is a quasi-norm (see Theorem 1.8 in \cite{Kam-Zyl}%
). Thus, we assume always here that $\alpha _{\Phi }>0$. In consequence, $%
\left( L^{\Phi },\Vert \cdot \Vert _{\Phi }\right) $ is a quasi-Banach ideal
space. Moreover, if additionally the function $\Phi $ is convex, then $%
\left( L^{\Phi },\Vert \cdot \Vert _{\Phi }\right) $ is a Banach ideal
space. Obviously, in the case when $\Phi \left( u\right) =u^{p},$ for $u\geq
0,$ with $p>0,$ the Orlicz space  $L^{\Phi }$ is just equal to Lebesgue space 
$L^{p}.$ 

We refer to \cite{Fo-Hu-Ko}, \cite{Kam-Mal-Per-2003} and \cite%
{Kam-Zyl} for more details, references and definitions we will use below.
The first two mentioned papers deal with more general Calder\'{o}n--Lozanovski%
\u{\i} spaces $E_{\Phi },$ which become the Orlicz spaces for $E=L^{1}.$

If we replace in Definition \ref{DefLHS}  the spaces $X_{k}=L^{p,r}\left( \Omega
_{k}\right) $ with $X_{k}=\left( L^{\Phi }\left( \Omega _{k}\right) ,\Vert
\cdot \Vert _{\Phi }\right) $ we get the notion of Orlicz-Herz spaces which
we will denote by $HO_{a,q}^{\Phi }.$ Consequently, by Proposition \ref{ideal}, the spaces $HO_{a,q}^{\Phi }$ are quasi-normed ideal spaces (normed ideal spaces
if $\Phi $ is convex). Moreover, the Orlicz spaces $L^{\Phi }$ have the
Fatou property (see Lemma 4.12 in \cite{Fo-Hu-Ko} or Lemma 2.2 in \cite%
{Kam-Mal-Per-2003}, for more general case), whence, by Corollary \ref{KH-complete}, $
HO_{a,q}^{\Phi }$ are quasi-Banach ideal spaces (Banach ideal spaces if $%
\Phi $ is convex). 

Recall that $L^{\Phi }\in \left( OC\right) $ if and only $\Phi $ satisfies
the growth condition $\Delta _{2}$ suitable to the measure $\mu $ (that is $%
\Phi $ satisfies the condition $\Delta _{2}$ for all arguments if $\mu
\left( \Omega \right) =\infty $ or $\Phi $ satisfies the condition $\Delta
_{2}$ for large arguments if $\mu \left( \Omega \right) <\infty $) - see
Theorem 5.7 in \cite{Fo-Hu-Ko} with $E=L^{1}$. Thus, applying Theorem \ref{oc}, Lemma \ref{dens} and Lemma \ref{sep}, we conclude immediately

\begin{corollary}
Let $a\in \mathbb{R},q>0.$ The Orlicz-Herz spaces $HO_{a,q}^{\Phi
}$ are order continuous if and only if they are separable if and only if $%
q<\infty $ and $\Phi $ satisfies the condition $\Delta _{2}$ suitable to the
measure $\mu $ (that is $\Phi $ satisfies the condition $\Delta _{2}$ for
all arguments if $\mu \left( \Omega \right) =\infty $ or $\Phi $ satisfies
the condition $\Delta _{2}$ for large arguments if $\mu \left( \Omega
\right) <\infty $). Furthermore, in that case, the set of simple functions having the finite measure is dense in the spaces $HO_{a,q}^{\Phi
}$.    
\end{corollary}

Recall also that the embedings between Orlicz spaces $L^{\Phi _{1}}$ and $%
L^{\Phi _{2}}$ depend on the respective inequalities between functions $\Phi
_{1}$ and $\Phi _{2}$ (see Theorem 2.3 from \cite{Kam-Mal-Per-2003}).
Therefore, we can conclude easily analogous result to Theorem \ref{wloz} (we omit
the formulation).

It is well known that if $\Phi $ is a convex Orlicz function satisfying the condition $\Delta _{2}$ suitable to the measure $\mu ,$ then 
\begin{equation*}
\left( L^{\Phi },\Vert \cdot \Vert _{\Phi }\right) ^{\ast }=\left( L^{\Phi
},\Vert \cdot \Vert _{\Phi }\right) ^{^{\prime }}=\left( L^{\Phi ^{\ast
}},\Vert \cdot \Vert _{\Phi ^{\ast }}^{O}\right) ,
\end{equation*}%
where $\Phi ^{\ast }$ is the complementary Orlicz function to $\Phi $ (in
the sense of Young) and $\Vert \cdot \Vert _{\Phi ^{\ast }}^{O}$ is the
Orlicz norm with respect to $\Phi ^{\ast },$ that is,%
\begin{equation*}
\Phi ^{\ast }\left( u\right) =\sup_{v\geq 0}\left\{ uv-\Phi \left( v\right)
\right\} \text{ and }\Vert f\Vert _{\Phi ^{\ast }}^{O}=\sup \left\{ \int
fg:I_{\Phi ^{\ast }}(g)\leq 1\right\} 
\end{equation*}
(see \cite{Fo-Hu-Ko}, \cite{Kam-Mal-Per-2003}, \cite{Kam-Zyl} for respective
references). Consequently, applying Corollary \ref{Kothe-dual}, we get

\begin{corollary} \label{Kothe-dual2} [\textbf{The K\"{o}the dual of Orlicz-Herz spaces}]
Denote by $\mathcal{X}=\left( \left( L^{\Phi }\left( \Omega _{k}\right)
,\Vert \cdot \Vert _{\Phi }\right) \right) _{k\in I}$ the family of Orlicz
spaces and set $\mathcal{X}^{\prime }=\left( \left( L^{\Phi ^{\ast }}\left(
\Omega _{k}\right) ,\Vert \cdot \Vert _{\Phi ^{\ast }}^{O}\right) \right)
_{k\in I}$. Let $E=l_{q}\left( 2^{ak}\right) $ with $a\in \mathbb{R}$ and $%
1<q<\infty .$ Then $E^{\prime }=l_{q^{\prime }}\left( 2^{-ak}\right) ,$
where $q^{\prime }$ is the conjugate exponents to $q.$ If $\Phi $ is a
convex function satisfying the condition $\Delta _{2}$ suitable to the
measure $\mu $, then%
\begin{equation*}
\left( HO_{a,q}^{\Phi },\left\Vert \cdot \right\Vert _{HO_{a,q}^{\Phi
}}\right) ^{\ast }=\left( HO_{a,q}^{\Phi },\left\Vert \cdot \right\Vert
_{HO_{a,q}^{\Phi }}\right) ^{^{\prime }}=\left( HO_{-a,q^{\prime }}^{\Phi
^{\ast }},\left\Vert \cdot \right\Vert _{E^{\prime }\left( \mathcal{X}%
^{\prime }\right) }\right) 
\end{equation*}%
with equivalent norms. Furthermore, in this case a H\"{o}lder type
inequality (as in Proposition \ref{HolderIq}) follows.
\end{corollary}

The elementary calculations show that 
\begin{equation*}
\left\Vert \chi _{A}\right\Vert _{\Phi }=\frac{1}{\Phi ^{-1}\left( 1/\mu
\left( A\right) \right) },
\end{equation*}%
where $\Phi ^{-1}$ is the generalized inverse of $\Phi .$ Furthermore, it is very well known that $\left\Vert \cdot \right\Vert _{\Phi }\sim \left\Vert \cdot
\right\Vert _{\Phi }^{O}.$ Thus, applying Theorem \ref{BFS1}, we get

\begin{corollary}
Denote by $q^{\prime }$ the conjugate exponents to $q$ with $1<q<\infty $.
Suppose $\Phi $ is a convex Orlicz function satisfying the condition $\Delta
_{2}$ suitable to the measure $\mu $. Then the Orlicz-Herz space $HO_{a,q}^{\Phi }$ is a Banach function space (in
the sense of Definition \ref{function norm}) if and only if for every $A\in
\Sigma $ with $\mu \left( A\right) <\infty $, the following conditions hold:
\begin{enumerate}
\item[(a)] $\sum\limits_{k\in I}2^{kaq}\left( \frac{1}{\Phi ^{-1}\left(
1/\mu \left( A\cap \Omega _{k}\right) \right) }\right) ^{q}<\infty $,

\item[(b)] $\sum\limits_{k\in I}2^{-kaq^{^{\prime }}}\left( \frac{1}{\left(
\Phi ^{\ast }\right) ^{-1}\left( 1/\mu \left( A\cap \Omega _{k}\right)
\right) }\right) ^{q^{\prime }}<\infty $.
\end{enumerate}
\end{corollary}


\begin{thebibliography}{99}

\bibitem{BMO} {\small Aykol, C., Guliyev, V.S., Serbetci, A., {\it Boundedness of the maximal operator in the local Morrey-Lorentz
spaces}, J. Inequal. Appl. {\bf 11} (2013), 2013:346, 11.}

\bibitem{Baernstein} {\small Baernstein, II, A. and Sawyer, E. T., {\it Embedding and multiplier theorems for {$H^{p}({\bf R}^n)$}}, Mem. Amer. Math. Soc. {\bf 53} (1985), iv+82.}

\bibitem{Sharpley1988} {\small Bennett, C. and Sharpley, R.,
{\it Interpolation of operators}, Vol. 129, Pure and Applied Mathematics, Academic
Press, Inc., Boston, MA, 1988.}

\bibitem{Bergh} {\small Bergh, J., Löfström, J., {\it Interpolation spaces. {A}n introduction}, Grundlehren der Mathematischen Wissenschaften, No. 223, Springer-Verlag, Berlin-New York,
1976.}


\bibitem{FSOLST} {\small Besoy, B.F., Cobos, F., {\it Function spaces of {L}orentz-{S}obolev type: atomic decompositions, characterizations in terms of wavelets, interpolation and multiplications}, J. Funct. Anal. {\bf 282} (2022), Paper No. 109452, 46.}

\bibitem{Caledron} {\small Calder\'{o}n, A.-P., {\it Intermediate spaces and interpolation, the complex method}, Studia Math. {\bf 24} (1964), 113--190.}

\bibitem{Coifman} {\small Coifman, R. R.; Cwikel, M.; Rochberg, R.; Sagher, Y.; Weiss, G.}, {\it A theory of complex interpolation for families of Banach spaces}
Adv. in Math. 43 (1982), no. 3, 203–229. 

\bibitem{SLPE} {\small Drihem,D., {\it Semilinear parabolic equations in Herz spaces}, Appl. Anal. (2022).}

\bibitem{Dr2} {\small Drihem,D., {\it Herz-type Sobolev spaces on domains}, Lee Matematische 77 (2022), no. 2, 229-263.}

\bibitem{Fefferman1979PAMS} {\small Fefferman, R.,{\it A note on singular integrals}, Proc. Amer. Math. Soc. {\bf 74} (1979), 266--270.}

\bibitem{SFT} {\small Feichtinger, H. G. and Weisz, F., {\it Herz spaces and summability of Fourier transforms}, Math. Nachr. {\bf 281} (2008), 309-324.}

\bibitem{Fo-Hu-Ko} {\small  Foralewski, P., Hudzik, H. and Kolwicz, P., {\it Quasi-modular spaces with applications to quasi-normed Calderón--Lozanovskiĭ spaces }, avalaible on arxiv : https://arxiv.org/abs/2208.08970.}

\bibitem{Garcia94} {\small García-Cuerva, J., Herrero, M.-J. L., {\it A theory of Hardy spaces associated to the Herz spaces}, Proc. London Math. Soc. (3) 69 (1994), no. 3, 605–628. }

\bibitem{Grafakos} {\small Grafakos, L., {\it Classical {F}ourier analysis}, Graduate Texts in Mathematics, Vol. 249, Springer, New York, 2014.}


\bibitem{ossc1} {\small Guliyev, V.S., Aliyev, S.S., Karaman, T., Shukurov, P.S., {\it Boundedness of sublinear operators and commutators on generalized {M}orrey spaces}, Integral Equations Operator
Theory. {\bf 71} (2011), 327--355.}

\bibitem{BCZO} {\small Guliyev, V.S., Aykol, C., Kucukaslan, A., Serbetci, A., {\it Maximal operator and {C}alderon-{Z}ygmund operators in local {M}orrey-{L}orentz spaces}, Integral Transforms Spec. Funct. {\bf 27} (2016), 866--877.}

\bibitem{dual} {\small Hern\'{a}ndez, E. and Yang, D., {\it Interpolation of {H}erz spaces and applications}, Math. Nachr. {\bf 205} (1999), 69--87.}

\bibitem{Herz} {\small Herz, C. S., {\it Lipschitz spaces and {B}ernstein's theorem on absolutely convergent {F}ourier transforms}, J. Math. Mech. {\bf 18} (1968/69), 283--323.}

\bibitem{Young} {\small Ho, K.-P., {\it Young's inequalities and {H}ausdorff-{Y}oung inequalities on {H}erz spaces}, Boll. Unione Mat. Ital. {\bf 11} (2018), 469--481.}

\bibitem{eqnorm} {\small Johnson, R., {\it Lipschitz spaces, {L}ittlewood-{P}aley spaces, and convoluteurs}, Proc. London Math. Soc. (3), {\bf 29} (1974), 127--141.}

\bibitem{KPR84} {\small Kalton, N. J. and Peck, N. T. and Roberts, James W., {\it An {$F$}-space sampler}, London Mathematical Society Lecture Note Series, Vol. 89, Cambridge Univ. Press, Cambridge, 1984. }

\bibitem{Kam-Mal-Per-2003} {\small  Kami\'{n}ska, A., Maligranda, L. and  Persson, L. E., {\it Indices, convexity and concavity of Calder\'{o}n--Lozanovski\u{\i} spaces }, Math. Scand.{\bf 92} (2003), 141-160.}

\bibitem{Kam-Zyl} {\small  Kami\'{n}ska, A. and  \.{Z}yluk, M., {\it Local geometric properties in quasi-normed Orlicz spaces}, Function spaces XII, Banach Center Publ., {\bf 119}, Polish Acad. Sci. Inst. Math., Warsaw, 2019, 197--221. }


\bibitem{Kam-Mal-Isr} {\small Kami\'{n}ska, A. and Maligranda, L., {\it On {L}orentz spaces {$\Gamma_{p,w}$}}, Israel J. Math. {\bf 140} (2004), 285--318.}

\bibitem{kantor} {\small Kantorovich, L. V. and Akilov, G. P., {\it Functional
Analysis}, ``Nauka'', Moscow, 1984.}

\bibitem{Kol2018-Posit} {\small Kolwicz, P., {\it Kadec-{K}lee properties of some quasi-{B}anach function spaces}, Positivity  {\bf 22} (2018), 983--1013.}

\bibitem{KLM-2019} {\small Kolwicz, P., Le\'{s}nik, K. and Maligranda, L.,
{\it Symmetrization, factorization and arithmetic of quasi-{B}anach function spaces}, J. Math. Anal. Appl. {\bf 470} (2019), 1136--1166.}



\bibitem{Holder-lorentz} {\small Kolyada, V. and Soria, J., {\it H\"{o}lder type inequalities in {L}orentz spaces}, Ann. Mat. Pura Appl. (4), {\bf 189} (2010), 523--538.}

\bibitem{KPS82} {\small Kre\u{\i}n, S. G. and Petunin, Ju. I. and Sem\"{e}nov, E. M., {\it Interpolation of Linear Operators}, ``Nauka'', Moscow, 1978.}

\bibitem{Krist} {\small Kristiansson, E.}, {\it Decreasing Rearrangement and Lorentz L(p, q) Spaces}, Master Thesis, Department of Mathematics, Lule{\aa} 
  University of Technology (2002), publication available online on the University website.

\bibitem{Ku-Lan} {\small Kutzarova, D. and Landes, T., {\it Nearly uniform convexity of infinite direct sums}, Indiana Univ. Math. J. {\bf 41} (1992), 915--926.}

\bibitem{lau} {\small Laustsen, N.J., {\it Matrix multiplication and composition of operators on the direct sum of an infinite sequence of Banach spaces}, Math. Proc. Cambridge Philos. Soc. {\bf 131} (2001) 165--183.}

\bibitem{lee}  {\small Lee, H. J., {\it Complex convexity and monotonicity in quasi-Banach lattices}, Israel J. Math., {\bf 159} (2007) 57--91. }

\bibitem{Le-Piazza} {\small Lef\`evre, P. and Rodr\'{\i}guez-Piazza, L.,
{\it Absolutely summing {C}arleson embeddings on {H}ardy spaces}, Adv. Math. {\bf 340} (2018), 528--587.}

\bibitem{sublinear1} {\small Li, X. and Yang, D., {\it Boundedness of some sublinear operators on {H}erz spaces}, Illinois J. Math. {\bf 40} (1996), 484--501.}

\bibitem{Lin} {\small Lin, P.-K., {\it K\"{o}the-{B}ochner function spaces},
Birkh\"{a}user Boston, Inc., Boston, MA, 2004.}

\bibitem{Lin-Tza-s} {\small Lindenstrauss, J. and Tzafriri, L., {\it Classical {B}anach spaces. {I}. Sequence spaces}, Ergebnisse der Mathematik und ihrer Grenzgebiete, Band 92, Springer-Verlag, Berlin-New York, 1977.}

\bibitem{LinTza} {\small Lindenstrauss, J. and Tzafriri, L., {\it Classical Banach spaces. {II}. Function spaces}, Ergebnisse der Mathematik und ihrer Grenzgebiete [Results in Mathematics and Related Areas], Springer-Verlag, Berlin-New York, 1979.}



\bibitem{Lorentz1} {\small Lorentz, G. G., {\it Some new functional spaces}, Ann.
of Math. (2) {\bf 51} (1950), 37--55.}

\bibitem{Lorentz2} {\small Lorentz, G. G., {\it On the theory of spaces $\Lambda 
$}, Pacific J. Math. {\bf 1} (1951), 411--429.}

\bibitem{Lu92} {\small Lu, S. Z., Yang, D. C.,  {\it The Littlewood-Paley function and $\phi$-transform characterizations of a new Hardy space HK2 associated with the Herz space } Studia Math. 101 (1992), no. 3, 285–298. } 

\bibitem{ossc2} {\small Lu, G., Lu, S., Yang, D., {\it Singular integrals and commutators on homogeneous groups}, Anal. Math. {\bf 28} (2002), 103--134.}

\bibitem{HTH1} {\small Lu, S. and Yang, D., {\it The weighted
Herz-type Hardy space and its applications} , Sci. China Ser. A {\bf 38} (1995), 662--673.}

\bibitem{HTSAA} {\small Lu, S., Yang, D., \& Hu, G., {\it Herz type spaces and
their applications. essay}, Beijing: Science Press, 2008.}

\bibitem{Ma04} {\small Maligranda, L., {\it Type, cotype and convexity properties of quasi-Banach spaces}, in: \textquotedblleft Banach and Function Spaces", Proc. of the Internat. Symp. on Banach and Function Spaces (Oct. 2--4, 2003, Kitakyushu-Japan), Editors M. Kato and L. Maligranda, Yokohama Publ. 2004, 83--120. }

\bibitem{Ma08} {\small Maligranda, L., {\it Tosio {A}oki (1910--1989)}, in:
\textquotedblleft Banach and Function Spaces", Proc. of the Internat. Symp.
on Banach and Function Spaces (14--17 Sept. 2006, Kitakyushu-Japan), Editors
M. Kato and L. Maligranda, Yokohama Publ. 2008, 1--23.}

\bibitem{NRZ} {\small Nafis, H., Rafeiro, H., Zaighum, M. A., {\it A note on the boundedness of sublinear operators on grand variable Herz spaces},
J. Ineq. Appl. (2020), Paper No. 1}

\bibitem{Pesa} {\small Pe\v{s}a, D., {\it Wiener-{L}uxemburg amalgam spaces},
J. Funct. Anal. {\bf 282} (2022), Paper No. 109270, 47.}

\bibitem{PDE} {\small Ragusa, M. A., {\it Homogeneous {H}erz spaces and regularity results}, Nonlinear Anal. {\bf 71} (2009), e1909--e1914.}

\bibitem{sizecondition} {\small Soria, F. and Weiss, G., {\it A remark
on singular integrals and power weights}, Indiana Univ. Math. J. {\bf 43} (1994), 187--204.}

\bibitem{Stein1957PAMS} {\small Stein, E. M., {\it Note on singular integrals},
Proc. Amer. Math. Soc. {\bf 8} (1957), 250--254.}

\bibitem{Stein-1970} {\small Stein, E. M., {\it Singular integrals and differentiability properties of functions}, Princeton Mathematical Series, No. 30 Princeton University Press, Princeton, N.J. 1970 xiv+290 pp.} 

\bibitem{Northholland} {\small Triebel, H., {\it Interpolation theory, function
spaces, differential operators}, Vol. 18, North-Holland Publishing Co., Amsterdam-New York, 1978.}

\bibitem{NSE} {\small Tsutsui, Y., {\it The navier-stokes equations and weak Herz spaces}, Adv. Differential Equations {\bf 16} (2011), 1049 – 1085.}

\bibitem{HTH2} {\small Wang, H. and Liu, Z., {\it The Herz-type Hardy
spaces with variable exponent and their applications}, Taiwanese J. Math. {\bf 16} (2012), 1363--1389.}

\bibitem{Wnuk} {\small Wnuk, W., {\it Banach Lattices with Order Continuous Norms}, Polish Scientific Publishers PWN, 1999.}
\end{thebibliography}
\end{document}